\documentclass[12pt]{iopart}

\expandafter\let\csname equation*\endcsname\relax
\expandafter\let\csname endequation*\endcsname\relax
\usepackage{amsmath}
\usepackage{algorithm,algorithmic}

\usepackage{changebar}
\usepackage{amsthm}
\usepackage{amssymb}
\usepackage{graphicx}
\usepackage{fullpage}
\usepackage{xcolor}
\usepackage{color}
\usepackage{caption}
\usepackage{subcaption}
\usepackage{listings}
\usepackage{color}
\usepackage{hyperref}
\usepackage{tikz}
\usepackage{graphicx,scalerel}
\usepackage{placeins}

\renewcommand{\vec}[1] {\ensuremath{\boldsymbol{#1}}}
\renewcommand{\mat}[1]{\mathbf{{#1}}}
\newcommand{\trace}{\mathrm{Tr}}
\newcommand{\R}{\mathbb{R}}
\newcommand{\cov}{\mathrm{Cov}}
\newcommand{\uu}{\vec{u}}
\newcommand{\m}{\vec{m}}
\newcommand{\lam}{\vec{\lambda}}
\newcommand{\thth}{\vec{\theta}}

\newcommand{\V}{\mat{V}}
\newcommand{\LL}{\mathcal{L}}
\newcommand{\dJdu}{\frac{\partial \hat{J}}{\partial \vec{u}}}
\newcommand{\dJdm}{\frac{\partial \hat{J}}{\partial \vec{m}}}
\newcommand{\dAdm}{\frac{\partial \mat{A}}{\partial \vec{m}}}

\newcommand{\Aobj}{\mathcal{A}}

\newtheorem{proposition}{Proposition}[section]

\newtheorem{remark}{Remark}[section]

\theoremstyle{definition}

\begin{document}
\title{Hyper-Differential Sensitivity Analysis for Inverse Problems Constrained by Partial Differential Equations}

\author{Isaac Sunseri, Joseph Hart, Bart van Bloemen Waanders, Alen Alexanderian}

\date{\today}
\begin{abstract}
High fidelity models used in many science and engineering applications couple multiple physical states and parameters. Inverse problems arise when a model parameter cannot be determined directly, but rather is estimated using (typically sparse and noisy) measurements of the states. The data is usually not sufficient to simultaneously inform all of the parameters. Consequently, the governing model typically contains parameters which are uncertain but must be specified for a complete model characterization necessary to invert for the parameters of interest. We refer to the combination of the additional model parameters (those which are not inverted for) and the measured data states as the``complementary parameters". We seek to quantify the relative importance of these complementary parameters to the solution of the inverse problem. To address this, we present a framework based on hyper-differential sensitivity analysis (HDSA). HDSA computes the derivative of the solution of an inverse problem with respect to complementary parameters. We present a mathematical framework for HDSA in large-scale PDE-constrained inverse problems and show how HDSA can be interpreted to give insight about the inverse problem. We demonstrate the effectiveness of the method on an inverse problem by estimating a permeability field, using pressure and concentration measurements, in a porous medium flow application with uncertainty in the boundary conditions, source injection, and diffusion coefficient.
\end{abstract}


\section{Introduction}

Rapid advances in numerical algorithms and computing infrastructure have made it feasible to simulate 
complex multiphysics systems governed by systems of partial differential equations (PDEs) on high resolution 
computational grids.
Inverse 
problems arise when some model
parameters cannot be determined directly, but rather are estimated using
measurements of the model state variables. The states
may correspond to different physical quantities with varying data volumes and
measurement fidelities, and measurements are typically sparse and 
noisy due to budget and hardware limitations.

Inverse problems governed by systems with complex physics involve various sources of
uncertainty. This includes the uncertainty in the parameters being estimated,
uncertainty in measurement data, and uncertainty in parameters in
governing PDEs that are not the focus of the parameter estimation, but are
needed for a full model specification.
For clarity, we refer to the model parameters being estimated as
\emph{inversion parameters} and to the other model parameters besides the
inversion parameter as \emph{auxiliary parameters}. Additionally, we refer to
the parameters specifying the experimental conditions, such as types of
measurements or measurement noise levels, as \emph{experimental parameters}.
The auxiliary parameters and the experimental parameters are needed for the
formulation of the inverse problem. We call the combination of auxiliary and
experimental parameters the \emph{complementary parameters}.  This article is
about understanding and quantifying the impact and relative importance of the
perturbations in complementary parameters to the solution of an inverse
problem.

For illustration, let us consider a subsurface flow problem, in which we seek
to invert for the log-permeability field using a tracer test. The forward model
we consider is given by the mass conservation, constrained  by Darcy's law,
resulting in a linear elliptic PDE governing the pressure, and a time-dependent
PDE governing diffusion and transport of the tracer.  The inversion parameter
here is the log-permeability field.  The auxiliary parameters include the
source terms (e.g., the tracer injection), boundary conditions, and
coefficients (e.g., the diffusion coefficient) in the governing PDE system.
The measurements correspond to the two states: pressure and concentration. The
experimental parameters correspond to noise in these measurements.

We propose a general framework to assess the relative importance of
complementary parameters in determining the solution of the inverse
problem. To do so, we build upon previous work
in~\cite{HartvanBloemenWaanders19}, and a series of related articles
\cite{Brandes06,griesse_constraints,Griesse_part_2,Griesse_part_1,Griesse_Thesis,Griesse_SISC,griesse_3d,Griesse_AD},
that introduced hyper-differential sensitivity analysis (HDSA) for
PDE-constrained optimization. HDSA computes the Fr\'echet derivative
of the solution of the inverse problem with respect to complementary
parameters. We use this derivative to define hyper-differential
sensitivities of the inverse problem solution with respect to the
complementary parameters.  These sensitivities describe the change in
the solution of the inverse problem with respect to perturbations of a
given parameter. We also define generalized sensitivity indices that
determine maximum (worst case) changes in the inverse problem solution
with respect to perturbations in a set of complementary parameters.

By providing sensitivity information on the experimental parameters, our
framework provides vital information for effective data collection. For
instance, by discovering the sensor measurements the inverse problem solution
is most sensitive to, we can identify the sensors where 
higher fidelity measurements are desired. 
This can be achieved by designing sensors with improved error
tolerances, or in problems where this is possible, repeating these sensor
measurements to reduce the associated measurement noise. 
Furthermore, we demonstrate that HDSA can be used to compare the relative importance 
of different types of sensors. 
As such, this process
complements experimental design which is used to design optimal sensor
placement for data measurements.  By calibrating the measurement
fidelities in a given experimental design (e.g., a sensor network), one can
make the most out of the measurements for effective parameter estimation.
Therefore, the proposed framework can be combined with an optimal experimental
design (OED) problem~\cite{Pazman86,Atkinson92,Ucinski05} to:
(i) identify an optimal set of
experiments; and (ii) calibrate the fidelities of the experiments or further
prune the specified set of experiments, based on the sensitivity analysis
results. 
HDSA provides a systematic framework to distinguish between measurements
obtained from different sensor types and understand the relative importance of
spatial and temporal sensor distributions. While HDSA is not intended to
replace OED, it augments it by providing unique insights into the influence of
various sensors in large-scale multiphysics applications.

Another important application of the proposed framework is guiding OED
under uncertainty.  In practical applications, an OED problem must be
found in such a way that it is robust with respect to uncertainty in
auxiliary model parameters; see e.g.,~\cite{KovalAlexanderianStadler19}. 
Performing sensitivity analysis of the
inverse problem solution with respect to auxiliary parameters informs
the sources of model uncertainty one needs to focus on when
solving an OED under uncertainty problem.  By focusing on sources of
model uncertainty the inverse problem solution is most sensitive to,
our framework can significantly reduce the complexity of an OED under
uncertainty problem. Furthermore, for applications in which multiple
experiments may be designed to target calibration of different
auxiliary parameters, the proposed sensitivities may identify where
experimental efforts should be invested to calibrate the most
influential parameters through a sequence of different experiments.

Additionally, in complex physics systems, typically the influence of various
sources of model uncertainty on the solution behavior is not clear a priori.
The proposed sensitivity analysis framework provides important insight about
the governing model.  

In contrast to traditional sensitivity
analysis, where one quantifies the contribution of auxiliary parameters to
variability in model output, our proposed framework provides a \emph{goal
oriented sensitivity analysis} approach by quantifying the impact of
perturbations in auxiliary parameters on estimation of unknown model
parameters. This enables determining which auxiliary parameters need to be
specified more accurately. In fact, it might be that some auxiliary parameters
should be estimated along with the inversion parameters, if possible.

Lastly, these sensitivities provide a computationally efficient low
order approach to uncertainty quantification for large-scale
systems. For instance, if the solution of the inverse problem must be
determined in real time to inform critical decision making, coupling
the estimated solution with a notion of uncertainty contributed by
errors in the complementary parameters provides real time uncertainty
estimation which is critical for making informed
decisions. 

The present work concerns local sensitivity analysis for deterministic
variational inverse problems.  We consider connections to statistical
formulations and global sensitivity analysis in
Section~\ref{sec:conclusion}. The contributions of this article are as follows:
\begin{itemize}

\item We define HDSA with respect to experimental parameters. This provides a systematic 
approach to compare different sensor types and reveals information about the relative 
importance of distributed (spatially and temporally) sensor measurements, neither of 
which can be easily determined by traditional OED.

\item Theoretical results are presented for linear inverse problems
to provide intuition and demonstrate properties of the sensitivities with respect to experimental parameters. 

\item We build upon previous work \cite{HartvanBloemenWaanders19} to develop a
more comprehensive mathematical framework for HDSA of nonlinear multiphysics
inverse problems. This is done, in particular, by maturing the idea of
generalized sensitivities as a tool for systematically comparing the importance
of auxiliary parameters (which may be of differing physical characteristics and
scales) alongside the novel development of HDSA for experimental parameters.

\item Comprehensive numerical results,
in a large-scale subsurface flow application,  
demonstrate the interpretation and use of HDSA for nonlinear 
multiphysics inverse problems.

\end{itemize}

The remainder of the article is organized as
follows. Section~\ref{sec:preliminaries} outlines the basic principles
of inverse problems and design of experiments. Section~\ref{sec:HDSA}
provides the mathematical formulation of hyper-differential
sensitivities and their interpretation for inverse problems
constrained by multiphysics. The computational considerations,
implementation, and cost analysis of HDSA is detailed in
Section~\ref{sec:comp_consider}. Section~\ref{sec:model_prob} presents
a large scale, multiphysics model problem, which is then used to
construct sensitivity results that are detailed in
Section~\ref{sec:sens_results}.  Concluding remarks and notes on
potential areas of future work are highlighted in
Section~\ref{sec:conclusion}.

\section{Preliminaries}
\label{sec:preliminaries}
In this section, we recall background material on inverse problems and design
of experiments, which are augmented by the proposed sensitivity
analysis in subsequent sections.  

\subsection{Inverse Problems}

In the present work, we are concerned with ill-posed inverse problems
governed by PDEs. Specifically, we seek to estimate a parameter 
$m$, henceforth called the inversion parameter, using data $\vec{y}$ and 
a model of the form 
\[
    F(m) + \vec\eta = \vec{y}. 
\]
Here $F$ is the parameter-to-observable map and $\vec\eta$
represents measurement noise. Evaluating $F(m)$ requires solving the 
governing PDEs and extracting the solution at measurement points.

Due to ill-posedness and availability of only sparse noisy measurements, 
we are led to variational formulations with suitable regularizations.
Specifically, we consider an optimization problem of the following form:
\begin{equation}\label{opt_prob}
\begin{aligned} 
&\min\limits_{u,m} J(u,m,\theta_e)  \\
\text{ s.t. } \ &v(u,m,\theta_a) = 0 \\
& u \in U, m \in \mathcal{M}.
\end{aligned} 
\end{equation}
Here, $U$ is an infinite dimensional reflexive Banach space containing the
state, $\mathcal{M}$ is a possibly infinite dimensional Hilbert space, 
$J$ is a regularized data misfit cost functional (we make this precise
below), and 
$v$ represents the 
constraining PDE system.  The experimental parameters, $\theta_e$, represent
uncertainty in the data, while $\theta_a$ are the auxiliary parameters
contained in the system of PDEs. Generally, solving this optimization problem
produces parameter estimates that are consistent with measurement data and the
model. For the remainder of the article, we refer to~\eqref{opt_prob} as the
inverse problem.

We mention that an alternative approach to address ill-posed inverse problems
is to consider a Bayesian formulation~\cite{Stuart10}.
In this approach, the inversion parameter $m$ is modeled as a random variable, and
the goal is to find a distribution law for $m$ that is consistent with
measurement data, the model, and a prior distribution of $m$ that models our
prior knowledge/beliefs about $m$.  In the present work, we
restrict our attention to deterministic formulation of inverse problems, as
described above.

We assume that the PDE represented by $v$
is uniquely solvable for any admissible $m$ and $\theta_a$. This
allows us to formulate \eqref{opt_prob} in \textit{reduced
  space} \cite{Akcelik_06}. Letting $\mathcal A(m,\theta_a)$ denote the solution operator for the PDE, i.e. $v(\mathcal A(m,\theta_a),m,\theta_a)=0$ for all $m$ and $\theta_a$, we define the reduced objective function $\hat{J}(m,\theta_e,\theta_a)=J(\mathcal A(m,\theta_a),m,\theta_e)$. In this article we focus on objective functions defined as a linear combination of data misfit and regularization, yielding a general form for the inverse problem
\begin{align} 
\label{inv_prob}
\min_{m \in \mathcal{M}} \hat{J}(m,\theta_e,\theta_a) := \frac12\| \mathcal Q\mathcal{A}(m,\theta_a) - \vec{y}(\theta_e)\|^2 + \alpha \mathcal{R}(m),
\end{align}
where $\vec{y}(\theta_e)$ is a vector of measured data (with uncertainty parameterized by $\theta_e$),  
$\mathcal Q$ an observation operator that maps the PDE solution to a set of 
observation locations, ${\mathcal R}$ is a regularization operator, 
and $\alpha$ is a regularization parameter. Traditional approaches to solving inverse problems fix $\theta=(\theta_e,\theta_a)$ to a best estimate and solve \eqref{inv_prob} by optimizing over $m$. Analyzing the influence of $\theta$ on the solution of \eqref{inv_prob} is the focus of this article.

Besides ill-posedness, such inverse problems are difficult to solve for a 
number of other reasons. These include having noisy observations, 
expensive forward PDE solves, tuning multiple experimental and 
modeling parameters, and optimization in infinite (or large finite) dimensional 
spaces. Common optimization methods used to tackle such problems 
include quasi-Newton, inexact 
Newton-CG, Gauss-Newton, and truncated CG trust region solvers. 
These optimization problems often require efficient gradient and Hessian computation 
through adjoint state methods, and repeated large scale linear system solves with
Krylov iterative methods. 
We direct the interested reader to a number of classical inverse problem 
references~\cite{BuiThanh12,Dashti17,Engl96,Ito96,Kaipio05,Tarantola05,Vogel02}.

\subsection{Design of Experiments}
An important aspect of solving an inverse problem is the collection of
informative measurement data, which is guided by \emph{design of 
experiments}. In our target inversion, this generally 
corresponds to specifying the locations of the sensors 
used to collect measurement data and is known as 
an optimal experimental design (OED) 
problem~\cite{Pazman86,Atkinson92,Ucinski05}. OED for inverse problems 
governed by differential equations has received significant attention 
in recent years; see e.g.,
\cite{
BauerBockKorkelEtAl00,
KorkelKostinaBockEtAl04,
HaberHoreshTenorio08,
BockKoerkelSchloeder13,
HuanMarzouk13,
LongScavinoTemponeEtAl13,
long2015fast,
HaberHoreshTenorio10,
HoreshHaberTenorio10,
AlexanderianPetraStadlerEtAl14,
AlexanderianPetraStadlerEtAl16,
AlexanderianSaibaba18}.
An OED problem is typically formulated with a statistical 
formulation of the inverse problem in mind. An optimal 
design is one that optimizes the statistical quality of the 
estimated parameters. 
Examples include
maximizing the expected information gain, leading to a D-optimal
design problem, or minimization of average posterior variance, 
leading to a Bayesian A-optimal design problem.

OED is a powerful tool that is used on a wide variety of problems.
It is also a very challenging problem both from mathematical and
computational point of view, especially when it comes to nonlinear
inverse problems governed by PDEs; see
e.g.,~\cite{HoreshHaberTenorio10,AlexanderianPetraStadlerEtAl16}.  The
developments in the present work are closely related to OED: while we
do not directly solve an OED problem, we address the following
relevant questions: (i) which measurements is the solution of an
inverse problem most sensitive to? And (ii) which measurement types
are most influential to the solution of the inverse problem? The
latter is tied to important questions typically not addressed in OED
literature: how should multi-purpose sensors that can take different
types of measurements be deployed, and how should different sensor
types be designed and deployed in an existing experimental design?

\section{Hyper-differential Sensitivity Analysis for Inverse Problems}
\label{sec:HDSA}
This section is devoted to our proposed framework for hyper-differential
sensitivity analysis (HDSA) of PDE-constrained inverse problems.  In
Subsection~\ref{subsec:math_form}, we detail the mathematical formulation of
the operator mapping complementary parameters to solutions of the
PDE-constrained inverse problem, and its Fr\'echet derivative.  In
Subsection~\ref{subsec:sens_ind} the hyper-differential sensitivities are
defined, as well as the generalized sensitivity index which is used to compare
the importance of sets of complementary parameters with different physical
characteristics.  Subsection ~\ref{subsec:linear_results} presents an
analytical result for linear inverse problems which connects the sensitivities
to the trace of the covariance in the solution of the inverse problem. 

\subsection{Mathematical Formulation}
\label{subsec:math_form}
HDSA differs from 
traditional sensitivity analysis in that it determines the sensitivity of the solution of an 
optimization problem rather than simply a model (which is typically a constraint 
in the optimization problem). We seek to perform HDSA on \eqref{inv_prob} to determine the sensitivity of the optimal $m$ to uncertainty in complementary (both experimental and auxiliary) parameters $\theta$ which are fixed when solving \eqref{inv_prob}.

HDSA uses the derivative of the solution of \eqref{inv_prob} 
with respect to $\theta$. To formally define HDSA, we assume that $\hat{J}$ is twice continuously differentiable with respect to $(m,\theta)$ and that $m^\star$ is a local minimum of \eqref{inv_prob} for specified complementary parameters $\theta=\theta^\star$. Assuming that the Hessian of $\hat{J}$ with respect
to $m$, evaluated at $(m^\star,\theta^\star)$, is positive definite \cite{HartvanBloemenWaanders19, Brandes06},
we can apply the implicit function theorem \cite[p.~38]{Ambrosetti95} to $\hat{J}_m$ (the Fr\'echet derivative of $\hat{J}$ with respect to $m$), to define 
a continuously differentiable mapping $\mathcal{F}$ 
from a neighborhood of $\theta^\star$ to a 
neighborhood of $m^\star$,
$$\mathcal{F}: \mathcal{N}(\theta^\star) \to \mathcal{N}(m^\star)$$ 
such that
$$ \hat{J}_m(\mathcal{F}(\theta),\theta) = 0, \quad \text{ for all } \quad \theta \in \mathcal{N}(\theta^\star),  $$
i.e., $\mathcal F$ maps complementary parameters to stationary points of \eqref{inv_prob}. The Fr\'echet derivative of $\mathcal{F}$ with respect to $\theta$, evaluated at $\theta^\star$, is given by
\begin{align}
\label{sen_op}
 \mathcal D := \mathcal{F}_\theta(\theta^\star) = -\mathcal{H}^{-1}\mathcal{B},
 \end{align}
where $\mathcal{H}$ is the Hessian of $\hat{J}$ with respect to $m$, evaluated at $m^\star$ and $\theta^\star$, 
i.e., $\mathcal H:=\hat{J}_{m,m}(m^\star,\theta^\star)$, and 
$\mathcal{B}$ is the Fr\'echet derivative of $\hat{J}_m$ with 
respect to $\theta$, evaluated at $m^\star$ and $\theta^\star$, i.e., $\mathcal B:=\hat{J}_{m,\theta}(m^\star,\theta^\star)$.

An intuitive interpretation of \eqref{sen_op} is that once \eqref{inv_prob} has
been solved to optimality for the specified $\theta^\star$, we take a
perturbation with respect to $\theta$ ($\mathcal B$) and apply a Newton step
($-\mathcal H^{-1}$) to update the solution of the inverse problem. We
interpret $\mathcal D \overline{\theta}$ as the sensitivity of the solution of
the inverse problem when the complementary parameters are perturbed in the
direction $\overline{\theta}$. Note that upon discretization, applying the
inverse of $\mathcal H$ to a vector requires a large-scale linear solve, which requires
many PDE solves. 

\subsection{Sensitivity Indices} \label{subsec:sens_ind}
We use \eqref{sen_op} to define sensitivity indices that attribute importance to each parameter. 
To this end, we first formalize assumptions about the parameter space.
In general, the complementary parameters $\theta = (\theta_e, \theta_a) \in \Theta$ take values 
in a possibly infinite dimensional space $\Theta = \Theta_1 \times \Theta_2 
\times ... \times \Theta_K$, which is a product of $K$ Hilbert spaces. 
The first $\ell$ parameter spaces contain the experimental parameters $\theta_e \in 
\Theta_1 \times ... \times \Theta_\ell$, while the remainder contain 
the auxiliary parameters $\theta_a \in \Theta_{\ell+1} \times ... \times \Theta_K$.
The product space $\Theta$ is equipped with the inner product
$$ \langle \theta, \phi \rangle_{\Theta} = \langle \theta_1, \phi_1 \rangle_{\Theta_1} + \dots + \langle \theta_K, \phi_K \rangle_{\Theta_K}, \quad \text{ for } \quad \theta, \phi \in \Theta. $$
We are particularly interested in cases where each $\Theta_k$, $k=1,2,\dots,K$, may have significantly different physical characteristics. For instance, corresponding to various physical quantities (thermal, fluid, solid, etc.) which have different spatial and temporal dependence.

To better understand spatial and temporal patterns of importance within a 
particular parameter or data source, we define pointwise sensitivity indices in space and time, and later generalized sensitivities which remove these units. From here on, we use 
$\mat{\Theta}$ and $\mat{\Theta}_k$ to denote the discretizations of the possibly 
infinite dimensional spaces $\Theta$ and $\Theta_k$. To respect spatiotemporal structure in discrete data, we use weighted norms corresponding to spatial and/or temporal discretizations. For instance, if $\vec{\theta}_k$ models perturbations of spatiotemporal data measurements then
$$ \|\vec{\theta}_k\|_{\mat{\Theta}_k} =
\sqrt{\frac{1}{n_tn_s}\sum_{i=1}^{n_tn_s}(\theta_k^i)^2}$$
where $\theta_k^i$ denotes the $i^{th}$ component of the vector $\vec{\theta}_k \in \mathbb R^{n_tn_s}$, $n_s$ and $n_t$ denote the number of spatial and temporal points, respectively.

Upon discretization of \eqref{opt_prob}, we let $\{\vec{b}_k^1,\vec{b}_k^2,\dots,\vec{b}_k^{n_k}\}$ denote 
a basis for each parameter space $\mat{\Theta}_k$,
where $n_k$ is the dimension of $\mat{\Theta}_k$. 

We define a basis for $\mat{\Theta}$ as $\{\vec{e}^i_k\}$ for $k = 1,\dots,K$ and $i = 1,\dots,n_k$ where

$$ \vec{e}_k^i = \begin{pmatrix}
\vec{0}_1 & \dots & \vec{0}_{k-1} & \vec{b}_k^i & \vec{0}_{k+1} & \dots & \vec{0}_K
\end{pmatrix}^{\top}.$$

We define the pointwise sensitivity indices for $k = 1,\dots,K$ and $i = 1,\dots,n_k$ as,
\begin{equation} 
\label{eqn:ind_sens}
	S_k^i = \frac{\| \mat{D}\vec{e}_k^i \|_{\mat{M}}}{\| \vec{e}_k^i \|_{\mat{\Theta}}},
\end{equation}
where $\mat{D}$ is the discretized sensitivity operator \eqref{sen_op}, and $\|\cdot\|_\mat{M}$ is the norm discretized consistently with respect to the norm in $\mathcal{M}$. 
The pointwise sensitivities measure the change in the solution of the inverse problem with 
respect to a perturbation of the $k^{th}$ parameter in direction $\vec{b}_k^i$.
Thus, high sensitivity indicates that errors 
in the parameter will cause a significant change in the reconstructed solution. 
This leads to an interpretation of the sensitivities as quantifying the importance of accurately
measuring or modeling the parameter. 

We would also like to compare sensitivities of various parameters with differing units to 
determine their importance relative to each other. 
To do this, we seek to formulate a generalized sensitivity index for a 
particular set of parameters
$\vec{\theta}_k \in \mat{\Theta}_k$, which can be interpreted 
as the maximum change in the solution with respect to a unit norm perturbation of 
the $k^{th}$ parameter.
Care must be taken here to ensure the sensitivities can be
compared accurately, as some parameters may be local in space and/or time, while others 
are not. We first consider a projection operator 
$\mat{T}_k: \mat{\Theta} \to \mat{\Theta}$, 
which zeros out all elements of $\vec{\theta}$ except those 
in $\mat{\Theta}_k$. Then the generalized sensitivity for parameters is
\begin{equation}
\label{eqn:gen_sens}
S_k = \max_{\vec{\theta} \in \mat{\Theta}}
\frac{\|\mat{D}\mat{T}_k\vec{\theta}\|_{\mat{M}}}{\|\vec{\theta}\|_{\mat{\Theta}}} .\\
\end{equation}
In this way, we obtain a single generalized sensitivity for each set of parameters $\vec{\theta}_k$, $k = 1,\dots,K$, which allows for a dimensionless comparison of parameters with different physical characteristics. 
 
\subsection{Interpretation of Experimental Parameter HDSA for Linear Inverse Problems}
\label{subsec:linear_results}
Sensitivity of auxiliary parameters, $\theta_a$ in the notation of this article, is a natural concept with a clear physical interpretation. The sensitivity of the solution of the inverse problem to experimental parameters is less intuitive, so we present an analytic result in Proposition~\ref{prp:lin_inv} to provide intuition. For conciseness and clarity in this subsection, we consider only uncertainty in the experimental parameters. 

Assume that $ \tilde{\vec{y}} = (\tilde{y}_1, \tilde{y}_2, \dots, \tilde{y}_n)^{\top}$ is a vector of noisy data which may be modeled by a linear parameter-to-observable map $\mathcal Q \mathcal A$ acting on an unknown parameter $m$. Estimating this unknown $m$ gives rise to a linear inverse problem.
To apply HDSA with respect to experimental parameters, we model the data as
$$y_i = \tilde{y}_i(1 + \theta_e^i) \qquad i=1,2,\dots,n,$$
where $(\theta_e^1, \theta_e^2, \dots, \theta_e^n)^{\top} = \vec{\theta_e} \sim
 N(0,\mat{\Sigma})$
is a perturbation of the nominal value $\tilde{y}_i$, i.e. a noise model.
Assuming uncorrelated observations, the noise covariance matrix is
$\mat{\Sigma} = \mathrm{diag}(\sigma_1^2, \ldots, \sigma_n^2)$. 
An estimate of the inversion parameter can be obtained by 
solving 
\begin{equation}\label{equ:lin_inv}
\min_{m} \hat{J}(m) :=
\frac12\| \mathcal Q \mathcal A m-\vec{y}\|_{\mat{\Sigma}^{-1}}^2 +
\frac{\alpha}{2}\|m\|_{R}^2
\end{equation} 
with $\vec{\theta_e}=0$, i.e. solving a linear least squares problem with the data $\tilde{\vec{y}}$. The norm in the regularization term is weighted by regularization
operator $R:\mathcal{M}\to\mathcal{M}$ which we assume is a self-adjoint, strictly 
positive linear operator on $\mathcal{M}$ \cite{Vogel02}. 

The estimator $m^\star$ obtained from solving \eqref{equ:lin_inv} is
a random variable, due to the random noise in the data. 
HDSA provides the sensitivity of its solution with respect to $\vec{\theta_e}=(\theta_e^1,\theta_e^2,\dots,\theta_e^n)$. This may be interpreted as the sensitivity of the least squares estimate with respect 
to the data, a metric to assess the relative importance of the observations. Such sensitivity information can be used to inform sensor designs and measurement 
tolerances. Proposition~\ref{prp:lin_inv} relates the variance of $m^\star$ (with respect to randomness in $\vec{\theta_e}$) to the pointwise sensitivities with respect to data measurements $S^i$ as defined 
in equation \eqref{eqn:ind_sens}. Here, we omit the subscript $(K=1)$ on the pointwise 
sensitivities, as we only consider one type of parameter in this section.

\begin{proposition}\label{prp:lin_inv}
$\trace(\cov(m^\star)) =  \sum_{i=1}^n (S^i)^2$, where $\trace$ denotes the trace of a linear operator and $S^i$ is defined as in \eqref{eqn:ind_sens} when $\Theta=\mathbb R^n$ is equipped with the $\mat{\Sigma}^{-1}$ weighted norm.
\end{proposition}
\begin{proof}
Computing the Fr\'echet derivative of the objective $\hat{J}$ in \eqref{equ:lin_inv}, setting it equal to zero, and solving for $m$ yields the solution of the inverse problem,
\[
m^\star = (\Aobj^*\mat{W}\Aobj + \alpha R)^{-1}\Aobj^*\mat{W}
\vec{y}=(\Aobj^*\mat{W}\Aobj + \alpha R)^{-1}\Aobj^*\mat{W} ( \tilde{\vec{y}} +
\tilde{\mat{Y}} \vec{\theta_e}),
\]
where $\Aobj^*$ denotes the adjoint of the linear operator $\Aobj$ and $\tilde{\mat{Y}}=\text{diag}\Big(\tilde{y_1},\tilde{y_2},\dots,\tilde{y_n} \Big)$.
Computing the Fr\'echet derivative of $m^\star$ with respect to $\vec{\theta_e}$ (which coincides with \eqref{sen_op}) yields
\[
\mathcal D = (\mathcal A^*\mat{W}\mathcal A + \alpha R)^{-1}\mathcal A^*\mat{W} \tilde{\mat{Y}}.
\]  
The covariance of the estimator $m^\star$ (with respect to randomness in $\vec{\theta_e}$) is  
\[
\cov(m^\star) = \cov(\mathcal D \vec{\theta_e})  = \mathcal D\mat{\Sigma} \mathcal D^*. 
\]
Therefore, 
\[
\begin{aligned}
\trace(\cov(m^\star)) &= \trace(\mathcal D\vec{\Sigma} \mathcal D^*) = \trace(\mathcal D^*\mathcal D\vec{\Sigma} ) = 
\sum_{i=1}^n\langle \vec{e}_i,\mathcal D^*\mathcal D\mat{\Sigma} \vec{e}_i\rangle \\
& = \sum_{i=1}^n (\sigma_i\| \mathcal D\vec{e}_i\|_\mathcal{M})^2 = \sum_{i=1}^n (\sigma_i \| \vec{e}_i \|_{\mat{\Sigma}^{-1}}S^i)^2 = \sum_{i=1}^n (S^i)^2,
\end{aligned}
\]
where $\vec e_i$ is the $i^{th}$ canonical unit vector in $\mathbb R^n$.
\end{proof}
This result provides useful intuition into the sensitivity indices and indicates
that the variance of the inverse problem solution is scaled by the magnitude 
of the sensitivities with respect to data. 

\begin{remark}
Notice that Proposition~\ref{prp:lin_inv} generalizes naturally to
Bayesian linear inverse problems. With a Gaussian prior and likelihood, the solution of the Bayesian linear inverse problem is a
Gaussian posterior. With the assumption 
$\vec{\theta}_e \sim N(0, \mat{\Sigma})$ on the measurement noise, and 
an appropriately chosen prior covariance, the maximum a posteriori (MAP)
point is equivalent to the solution $m^\star$ of the deterministic linear inverse 
problem. Taking the trace of the MAP point's covariance we again 
conclude Proposition~\ref{prp:lin_inv}. 
Note that in the Bayesian setting, we consider the average variance of the
MAP estimator as a measure of robustness of this point estimator for the 
inversion parameter. This is different than the average posterior uncertainty in 
the parameter given by the trace of the posterior covariance operator. 
\end{remark}

\section{Computational Considerations}
\label{sec:comp_consider} 
We begin our discussion with a simple 
illustrative example. 
To highlight 
the
dimensions of the discretized operators
and gain insight into the 
computational complexity of HDSA, we consider 
a discretized inverse problem with only auxiliary parameters: 
\begin{equation}\label{equ:ip_discrete} 
\begin{aligned}
\min_{\vec{m}} \hat{J}(\vec{m}) &= \frac12\|\mat{Q}\vec{u}-\vec{y}\|_\mat{W}^2 +
\frac{\alpha}{2}\|\vec{m}\|_\mat{R}^2 \\
\text{where } \  \mat{L}(\vec{m})\vec{u} &= \mat{V}\thth.
\end{aligned}
\end{equation}
where $\mat{Q}\in \R^{d\times n}$ is an observation operator, $\vec{u}\in\R^n$ the state 
vector,
$\vec{y}\in\R^d$ is the vector of experimental observations,
$\mat{W}\in\R^{d\times d}$ a symmetric weight matrix, $\alpha > 0$ a regularization
coefficient, $\vec{m}\in\R^p$ the discretized inversion parameter,
$\mat{R}\in\R^{p\times p}$ a symmetric positive definite regularization operator,
$\mat{L}(\vec{m})\in \R^{n\times n}$ a discretized differential operator,
$\vec{\theta}\in\R^k$ the vector of auxiliary parameters, and
$\mat{V}\in\R^{n\times k}$. Note that the discretized 
state dimension $n$ corresponds to the number of degrees of freedom in the mesh (typically large), and for problems with distributed parameters, $p$ will also have a comparable dimension to $n$ (frequently equal). The dimension of the auxiliary parameters, $k$, can also be large, potentially larger than $n$ when there are multiple distributed auxiliary parameters.

In practice, we compute the action of the gradient and Hessian of $\hat{J}$ using
a formal Lagrangian approach where each application of the Hessian requires two linear
PDE solves (inverting the matrix $\mat{L}(m)$). To compute the action of $\mat{D}$ (the discretization of \eqref{sen_op}) on a vector, we require 2 linear solves to apply the matrix $\mat{B}$ to a vector, and then $2I$ linear solves to apply $\mat{H}^{-1}$ to the resulting vector, where $I$ is the number of iterations needed by an iterative linear solver.
We direct the reader to~\ref{mtd:adj_op}, where we demonstrate the
adjoint method used to compute the gradient and Hessian of the reduced
objective function $\hat{J}$, as well as the operator $\mat{B}$. 

In general, the sensitivity indices \eqref{eqn:ind_sens} and generalized
sensitivity indices \eqref{eqn:gen_sens} may be computed in a variety of ways.
The efficiency of different approaches depends upon (i) the dimension of the
parameter space, (ii) the computational cost of the PDE solves, (iii) the
structure of the Fr\'echet derivative $\mathcal D$, and (iv) the available
computational resources. 

The computational bottleneck when computing \eqref{eqn:ind_sens} and
\eqref{eqn:gen_sens} is repeatedly inverting $\mathcal H$ (a large linear
system solve). In general, we are interested in systems of nonlinear PDEs. For
such systems, each application of $\mathcal H^{-1}$ requires $2I$ linearized
PDE solves, where $I$ is the number of iterations required by the linear solver
(such as conjugate gradient). Because HDSA is post-optimality analysis, we do
not require solving nonlinear systems repeatedly as in the inverse problem, but
rather solving PDEs which are linearized about the solution of the inverse
problem.

As introduced in \cite{HartvanBloemenWaanders19}, a randomized generalized
eigenvalue problem may be formulated to estimate the truncated generalized
singular value decomposition (GSVD) of $\mathcal D$. When the parameter
dimension is large and Fr\'echet derivative $\mathcal D$ is low rank, both
\eqref{eqn:ind_sens} and \eqref{eqn:gen_sens} may be efficiently estimated by
using the truncated GSVD and leveraging the parallelism of randomized methods.
We refer the reader to \cite{HartvanBloemenWaanders19} for additional details.

If $\mathcal D$ is not low rank but the parameter dimension and cost per PDE
solve is mild, we may compute \eqref{eqn:ind_sens} and \eqref{eqn:gen_sens}
directly by applying $\mathcal D$ to each basis function in $\Theta$. This does
not exploit structure as in the GSVD approach, but it is embarrassingly
parallel making it feasible for moderate parameter dimensions.

If $\mathcal D$ is not low rank and the parameter dimension or cost per PDE solver prohibits computing \eqref{eqn:ind_sens} and \eqref{eqn:gen_sens} directly, we may still compute \eqref{eqn:gen_sens} using a GSVD. Each generalized sensitivity index corresponds to the leading singular value of $\mathcal D$ acting on a projection operator. Since the number of generalized sensitivities are typically small, they may be estimated by using randomized solvers to compute the leading singular value. By exploiting parallelism, this may be done with a modest number of linearized PDE solvers regardless of the spectral decay in $\mathcal D$.

\section{Model Problem}
\label{sec:model_prob}
In this section, we present a multiphysics model problem
which is then used in Section~\ref{sec:sens_results} to compute hyper-differential sensitivities 
and demonstrate the usefulness and flexibility of HDSA. 
As a motivating example, we consider the problem of identifying the permeability 
field of a porous subsurface region with a tracer substance flowing through the domain. 
We consider a unit square domain $\Omega$ with 
boundary $\Gamma = \cup_{i=0}^3 \Gamma_i$, where 
$\Gamma_0$, $\Gamma_1$, $\Gamma_2$, and $\Gamma_3$ 
denote the bottom, right, top, and left edges of $\Omega$, respectively. 

We model subsurface flow of a fluid through a porous medium with Darcy's Law
and consider transport of the tracer through the medium governed by the
advection diffusion equation: 
\begin{subequations}\label{equ:pde}
\begin{align} 
-\nabla\cdot(e^m \nabla p) &= 0 \quad 
\text{ in } \Omega \label{equ:pressure}\\
c_t - \nabla\cdot\Big( \epsilon \nabla c \Big) + \nabla\cdot\big( \vec{v} c\big)
&= g \quad \text{ in } [0,T]\times\Omega \label{equ:conc}\\
p &= p_1 \quad \text{on } \Gamma_1 \\
p &= p_2 \quad \text{on } \Gamma_3 \\
\nabla p \cdot n &= 0 \quad \text{ on } \Gamma_0 \cup \Gamma_2 \\
\nabla c \cdot n &= 0 \quad \text{ on } [0,T]\times\{\Gamma_0\cup\Gamma_1\cup\Gamma_2\cup\Gamma_3\} \\
c(0,\cdot) &= 0 \quad \text{ in } \Omega 
\end{align} 
\end{subequations}

Here $p$ denotes the pressure field, $m$
the log-permeability field of the medium, 
$\vec{v} = -e^m\nabla p$ the Darcy velocity,
$c(t,x)$ the tracer concentration, $\epsilon$   
the diffusivity constant, and $g$ 
the source term of the injected tracer.
In the present example, we used $\epsilon = 0.025$.
For simplicity of notation, the constant fluid viscosity and constant 
porosity of the medium have been removed from the model. 
The Dirichlet pressure boundary conditions (on left and right 
boundaries) are described by non-zero fuctions
$p_1$ and $p_2$. We let $p_1$ be 
greater in magnitude than $p_2$, as this pressure difference 
will drive fluid flow from right to left through the domain,  
\[
\begin{aligned}
p_1(y) &= 15 + \cos(2 \pi y ) + \frac12\cos(4 \pi y), \\
p_2(y) &= 10 + 2\cos(2 \pi y).
\end{aligned}
\]
The tracer source is described by
$$g(x, y) = \sum_{k=1}^{16} 10e^{-100((x-v_k)^2 + (y-w_k)^2)}$$
where the source injection locations $(v_k, w_k)$ are arranged in a $4\times4$ grid as depicted by the diamonds in 
Figure~\ref{fig:source_loc}.

\begin{figure} [ht!]
	\centering
	\includegraphics[width=0.45\linewidth]{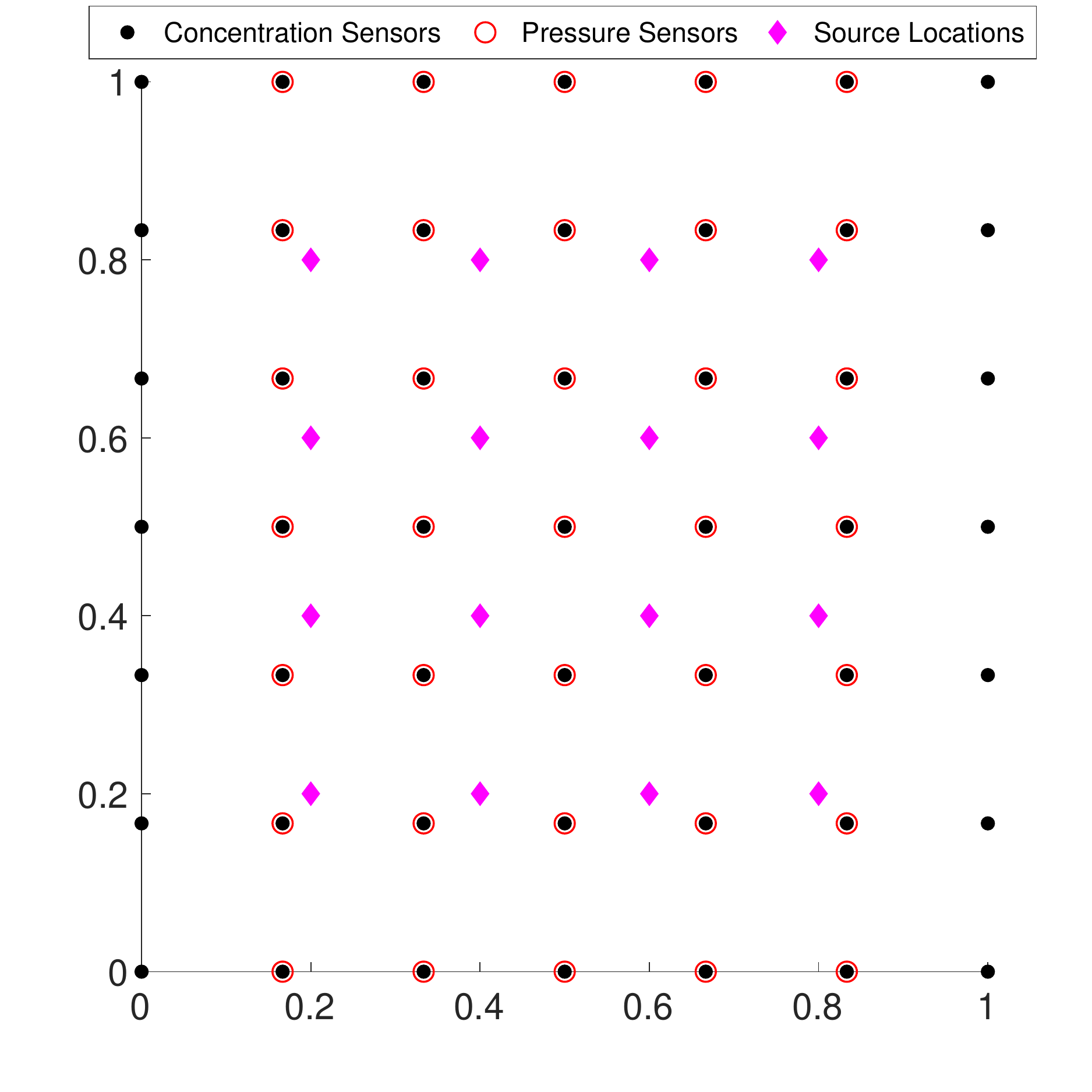}
	\caption{Concentration sensor, pressure sensor, and source locations.}
	\label{fig:source_loc}
\end{figure}
\FloatBarrier

We seek to solve an inverse problem to reconstruct the log-permeability field
$m$, using pressure and concentration 
measurements. Let $\mathcal Q$ denote the 
observation operator and
$\vec{y} \in \mathbb{R}^{n}$ be a vector of $n_p$ pressure measurements and $n_c$ concentration measurements at 
$n_t$ measurement times, giving a total of
 $n = n_p + n_cn_t$ data points,
$$ \vec{y} = \begin{bmatrix}
p_1 & p_2 & \dots & p_{n_p} & c_1 & c_2 & \dots & c_{n_cn_t} \end{bmatrix}. $$ 
The observation (sensor) locations are depicted in Figure~\ref{fig:source_loc}.

We consider the inverse problem
$$ \min_{m} \hat{J}(m) := \ \frac12\| \mathcal Q \Aobj(m)-\vec{y}\|_{\mat{W}}^2 
+ \frac{\alpha}{2}\int_{\Omega}\|\nabla m\|_2^2 \ dx $$
where $\mathcal A$ is the solution operator for \eqref{equ:pde}, 
$$ \mat{W} = \begin{pmatrix}
\frac{1}{\overline{p}^2\sigma^2} \mat{I}_{n_p} & 0 \\
0 & \frac{1}{\overline{c}^2\sigma^2} \mat{I}_{n_cn_t} 
\end{pmatrix} $$ 
is a data misfit weight matrix and $\alpha$ is the regularization coefficient.
We used $\alpha = 3 \times 10^{-2}$ in our numerical experiments; this was
chosen based on numerical experimentations seeking a 
regularization coefficient that is large enough to mitigate ill-posedness 
and at the same time produces a reasonable parameter reconstruction.
The weight matrix divides each measurement by the measurement noise 
$\sigma$ and the average of 
its data type ($\overline{p}$ and $\overline{c}$ respectively) to ensure 
the two data types, which are on different scales, have equivalent importance in the 
data misfit term.

We synthesize data for this problem with additive Gaussian noise that perturbs the data with
a standard deviation of 3$\%$ of the true value, i.e. $\theta_e^i \sim \mathcal{N}(0,\sigma), 
\text{ with } \sigma = 0.03$. Note that we assume pressure and concentration sensors 
have
the same proportional measurement error, $\sigma_p = \sigma_c = 0.03$. 

The inverse problem is solved on a 55$\times$55 finite element spacial discretization with 
48 time steps, while the data is generated from a forward PDE solve with a 109$\times$109 
finite element spacial discretization and 98 time steps.
We use a completely uninformed, constant 0 initial guess, with 49 concentration sensors 
and 35 pressure sensors arrayed throughout the domain as depicted in Figure~\ref{fig:source_loc} to solve the inverse problem. 

\section{Computational Results}
\label{sec:sens_results}

Using the model presented in Section~\ref{sec:model_prob}, we solve the inverse problem and
compute sensitivity indices to determine which parameters and data sources are most important.
Figure \ref{fig:permfield} depicts the 
true log permeability field we seek to reconstruct through the inverse problem, and 
the reconstructed solution found by solving the inverse problem with a 
truncated CG trust region solver. Note that since the optimization problem is non-convex we are only guaranteed to find a local minimizer.

\begin{figure} [ht!]
	\centering
	\includegraphics[width=0.7\linewidth]{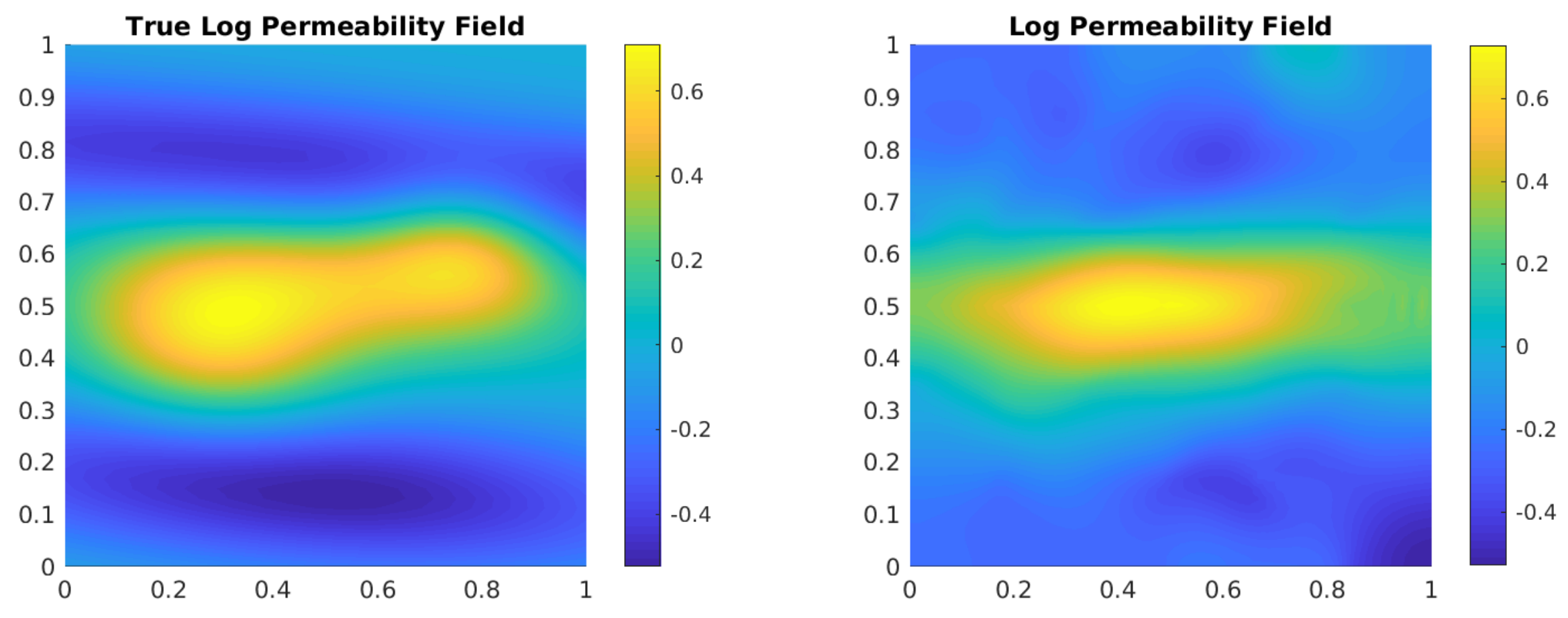}
	\caption{Left: True Permeability $\quad$ Right: Reconstructed Solution.}
	\label{fig:permfield}
\end{figure}
\FloatBarrier

In Subsection~\ref{subsec:gen_HDSA} we display and interpret generalized sensitivity indices 
for the pressure data, 
concentration data, tracer source term, diffusion coefficient, and the left and right pressure 
Dirichlet boundary conditions. In Subsection~\ref{subsec:exp_HDSA} we analyze the 
pointwise sensitivities with
respect to the experimental parameters (pressure and concentration data), and in Subsection~\ref{subsec:aux_HDSA} we analyze the pointwise 
sensitivities with respect to auxiliary parameters 
(source term, pressure Dirichlet boundary conditions, and diffusion coefficient). General discussion
of the importance and interpretation of the sensitivities is presented in Subsection~\ref{subsec:disc}. 

We model uncertain parameters as a nominal value times a parameterized 
perturbation. When our parameters of interest are constants, such as data measurements or 
modeling coefficients, we can model uncertain parameters as
\begin{equation*}
\label{eqn:param_perturb}
d = \tilde{d}(1+a\theta),
\end{equation*}
where $d$ is the parameter of interest, $\tilde{d}$ is the nominal value, 
$a$ is a scaling coefficient,
and $\theta \in [-1,1]$ the parameterization of the perturbation. 
In practice, we compute the sensitivity with $\theta = 0$, which corresponds to computing the
sensitivity at the nominal parameter value $\tilde{d}$. Extensions to global sensitivity analysis may consider sampling $\theta$ in $[-1,1]$.
The scaling coefficient $a$
is problem dependent, and should be set based on prior knowledge of the level of uncertainty in the
parameter of interest. When the parameter is a spatially and/or temporally distributed, we model uncertainty in the function using a
linear combination of basis functions
$$ f(x) = \tilde{f}(x)\Big(1+a\sum_{i=1}^{L}\theta_i\phi_i(x) \Big) $$
where $\tilde{f}$ its nominal estimate, $L$ the
dimension of the discretized basis, and $\{\phi_i\}$ are basis functions. The results in this article take
$\phi_i$ as linear finite element basis functions defined on a coarser mesh than the PDE is solved on (to enforce smoothness in perturbations).
For this model problem we let $a = 0.05$ for the 
experimental parameters and $a = 0.2$ for the auxiliary parameters which indicates
$5\%$ uncertainty in the data (experimental parameters) and $20\%$ uncertainty in the auxiliary parameters.

\subsection{Generalized Sensitivity Results} \label{subsec:gen_HDSA}
We calculate the generalized sensitivities for each parameter type, which allows for 
comparison between their relative importance. These 
are plotted in Figure~\ref{fig:Gen_sensitivities}.

\begin{figure} [ht!]
	\centering
	\includegraphics[width=0.48\linewidth]{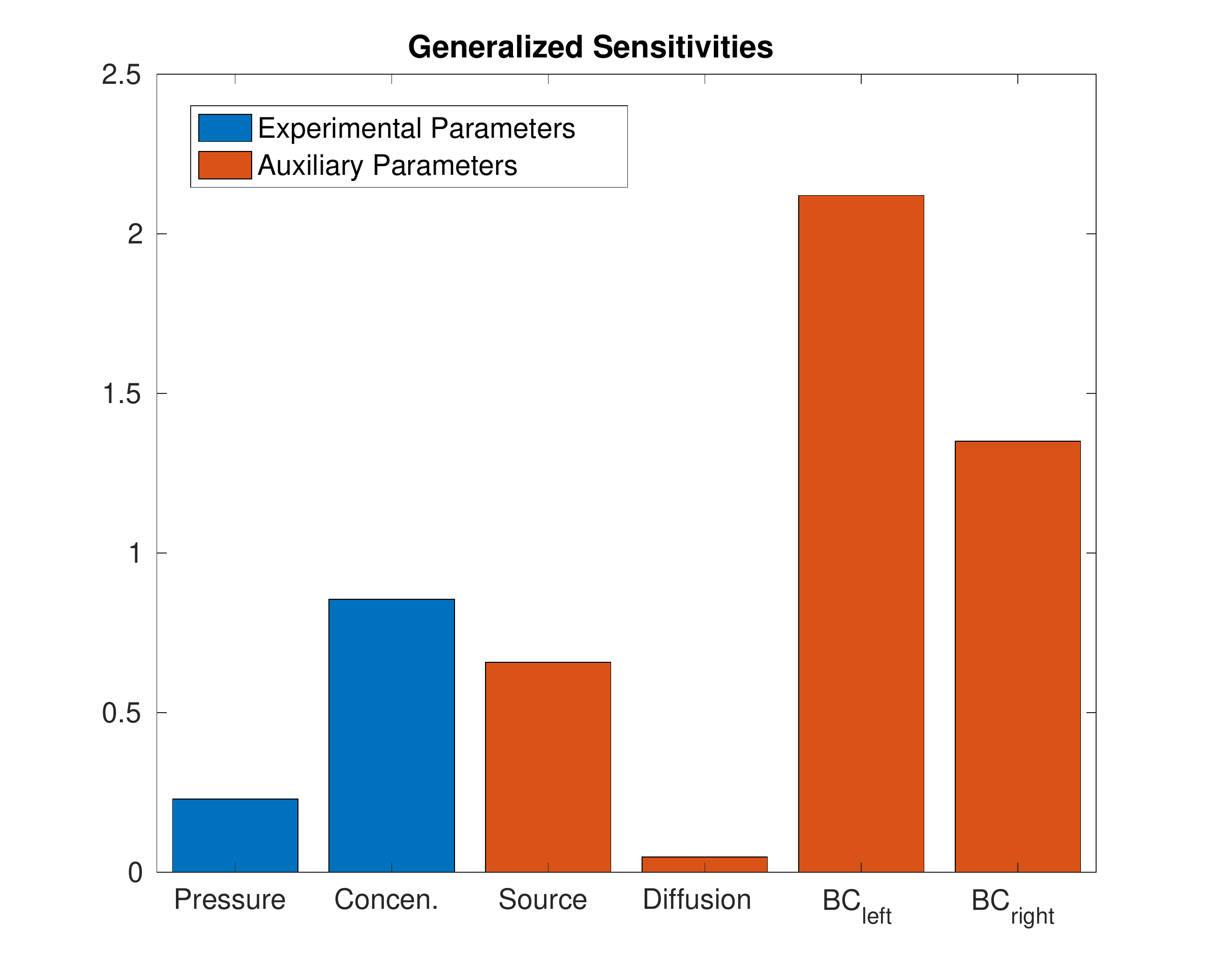}
	\caption{Bar graph of generalized parameter sensitivities.}
	\label{fig:Gen_sensitivities}
\end{figure}
\FloatBarrier

From Figure~\ref{fig:Gen_sensitivities}, we can see that for this specific model problem it is
most important to accurately measure the left and right boundary conditions. This makes sense 
intuitively, as the boundary conditions drive the fluid flow and the problem is advection dominated. 
We can also tell that
in terms of data collection, it is more important to accurately measure concentration than pressure
and that measuring the diffusion coefficient with a high degree of accuracy is relatively unimportant.

We consider the left boundary condition as an example and illustrate the interpretation of its generalized sensitivity index. Such principles of interpretation may be extended to any other generalized sensitivity indices but are omitted for conciseness. Figure~\ref{fig:perturbation} displays the true left boundary condition and the perturbation of the
left boundary condition corresponding to the generalized sensitivity index (the argument of the maximization in \eqref{eqn:gen_sens}. The perturbation plotted in
Figure~\ref{fig:perturbation} is the unit norm perturbation that results in the maximum change
in the inverse problem's solution. Thus, the generalized sensitivity index 
$S_5 = 2.12$ indicates that this unit norm perturbation will result in
a change of about 2.12 in the norm of the solution of the inverse problem. Scaling by the norm of the solution of the inverse problem gives an interpretation that the unit norm perturbation shown in Figure~\ref{fig:perturbation} results in approximately a 
$14\%$ change in the solution of the inverse problem. With this interpretation, a user may associate a level of uncertainty in the boundary condition and the resulting change in the estimated permeability field to determine if further calibration is needed.

\begin{figure} [ht!]
	\centering
	\includegraphics[width=0.4\linewidth]{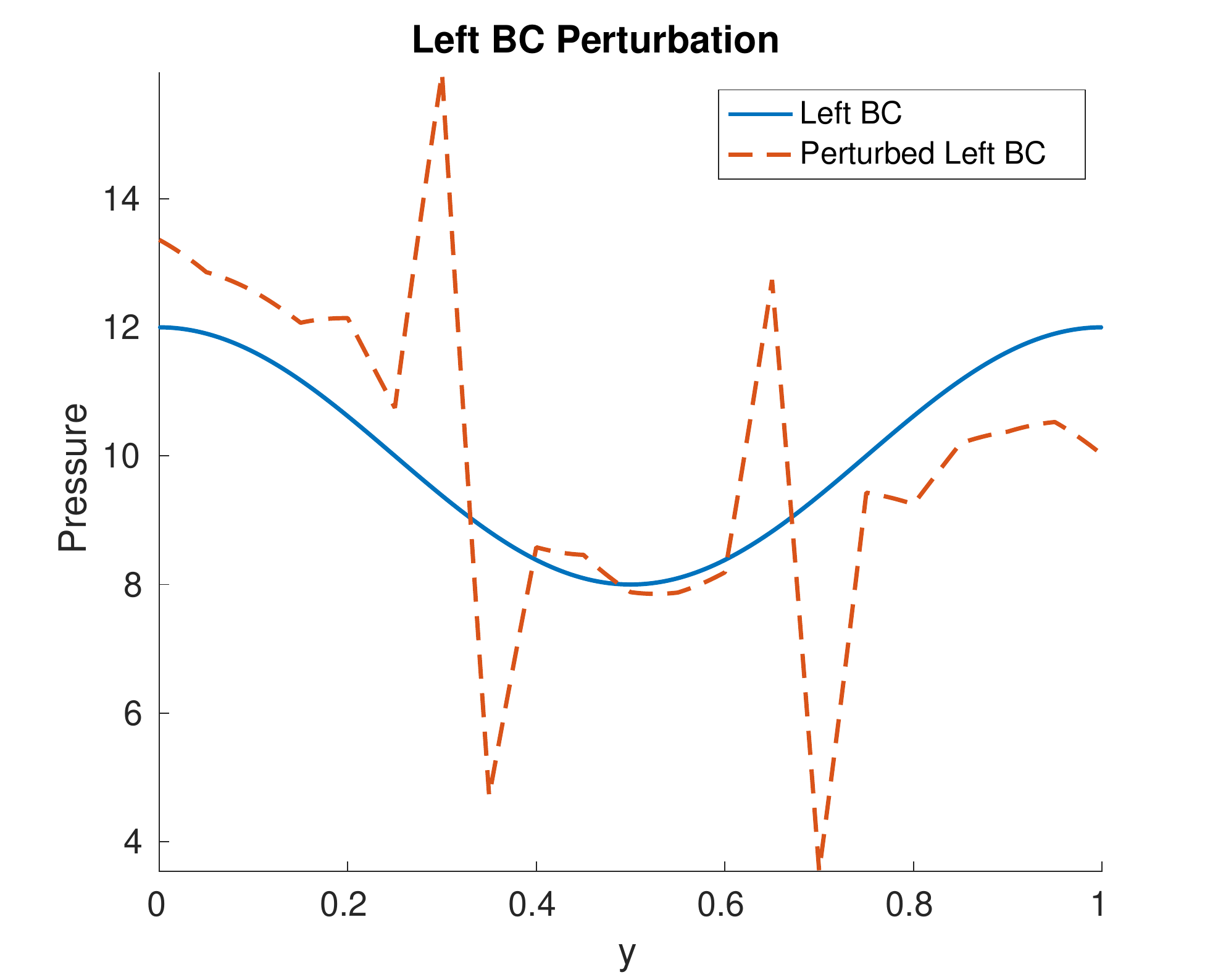}
	\caption{Perturbation of the left pressure Dirichlet boundary condition.}
	\label{fig:perturbation}
\end{figure}
\FloatBarrier

\subsection{HDSA with Respect to Experimental Parameters} \label{subsec:exp_HDSA}
In this section, we turn to the pointwise hyper-differential sensitivities \eqref{eqn:ind_sens} 
to study the spatial and temporal dependence within the experimental parameters. 
Using the reconstructed log-permeability field, we compute the sensitivities of the solution with
respect to both pressure and contaminant measurements at each sensor location, and for concentration, each time step. Figure 
\ref{fig:Contaminant_sensitivities} shows the spatial distribution of contaminant sensitivities 
(depicted by colored points using the right colorbar scale) at 
informative time snapshots, overlaid atop the tracer concentration field 
(depicted by a greyscale concentration map using the left colorbar scale). 
By overlaying these plots, we are able to study how the sensitivities 
relate to the tracer advection. 

\begin{figure} [ht!]
	\centering
	\includegraphics[width=0.47\linewidth]{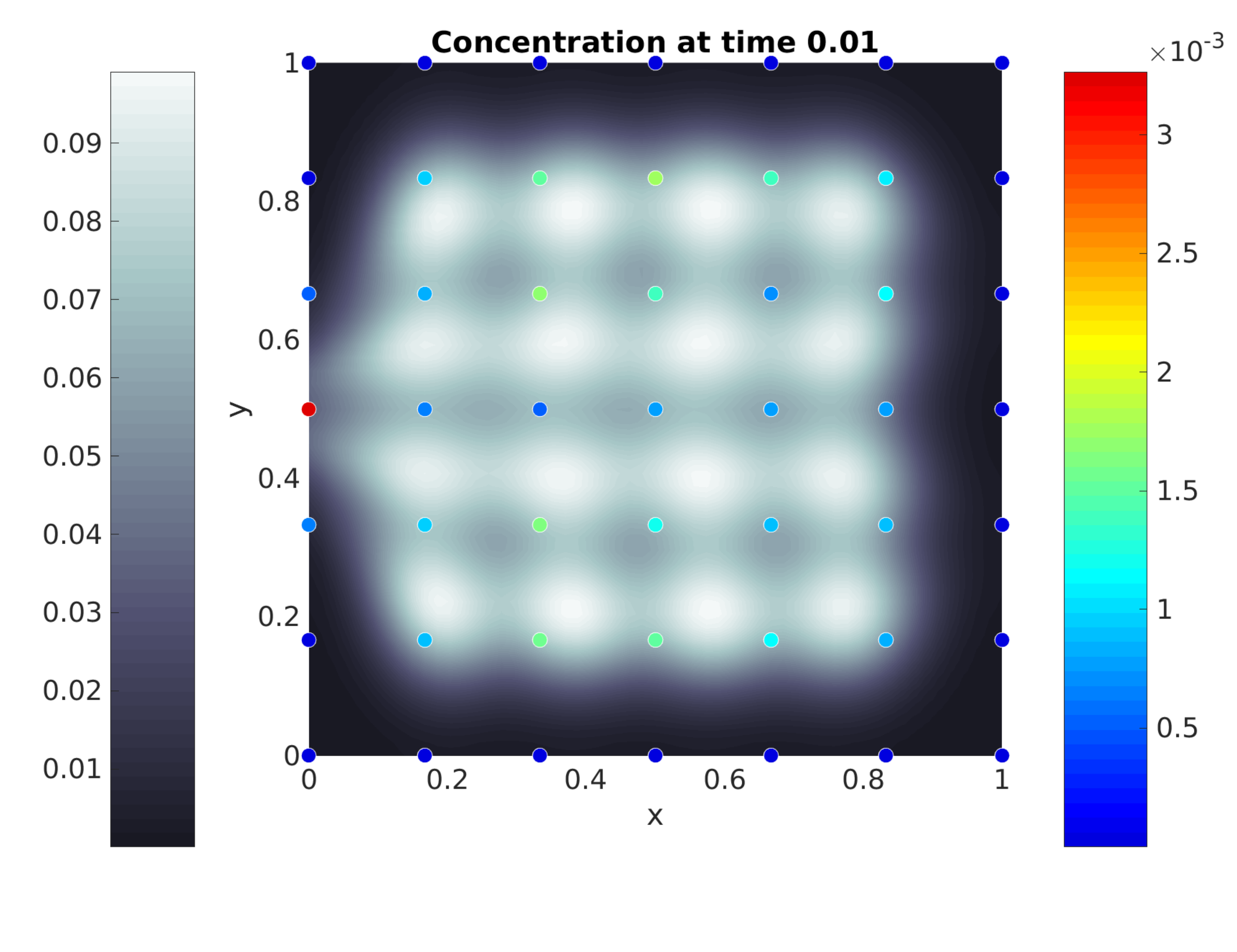}
	\includegraphics[width=0.47\linewidth]{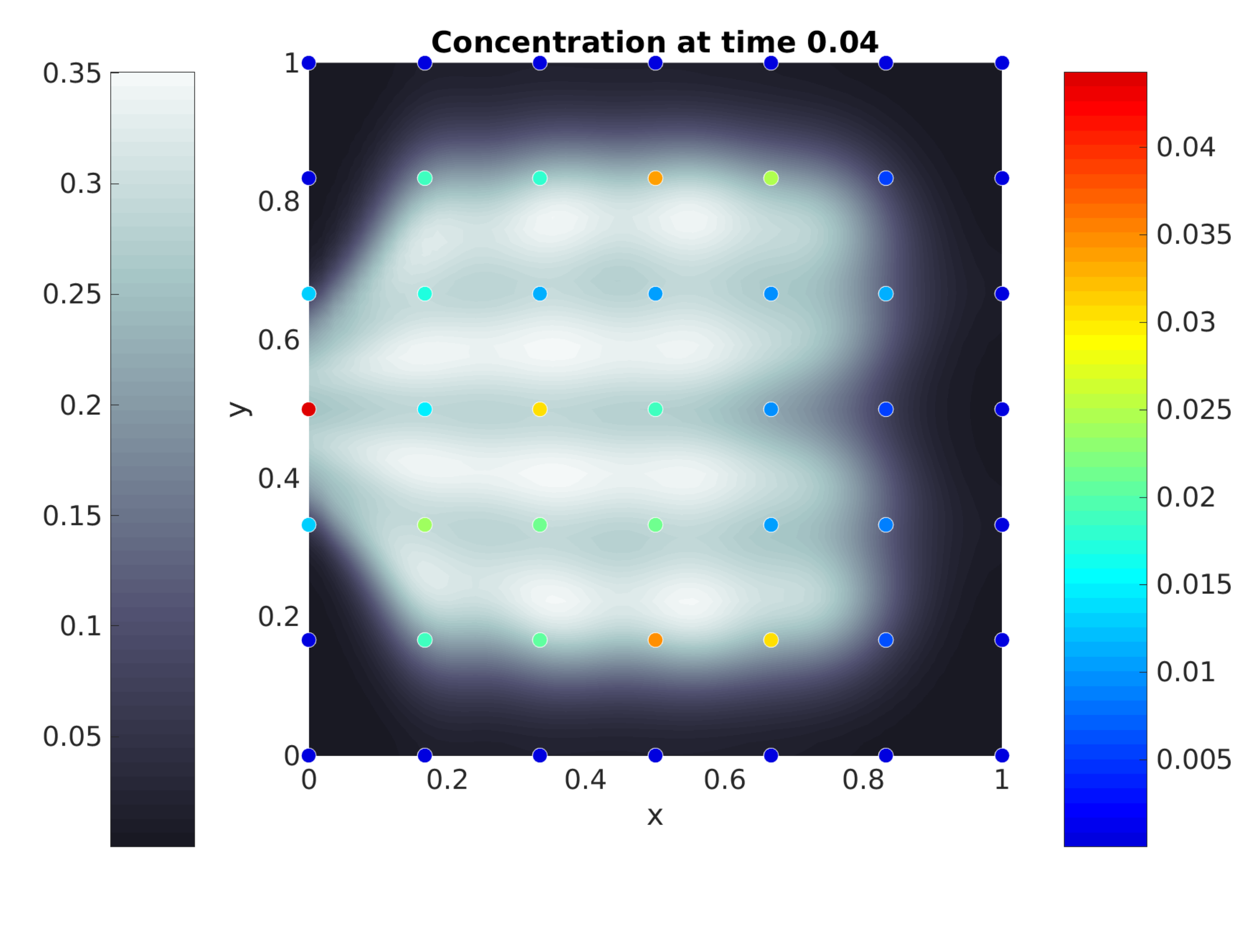}
	\includegraphics[width=0.47\linewidth]{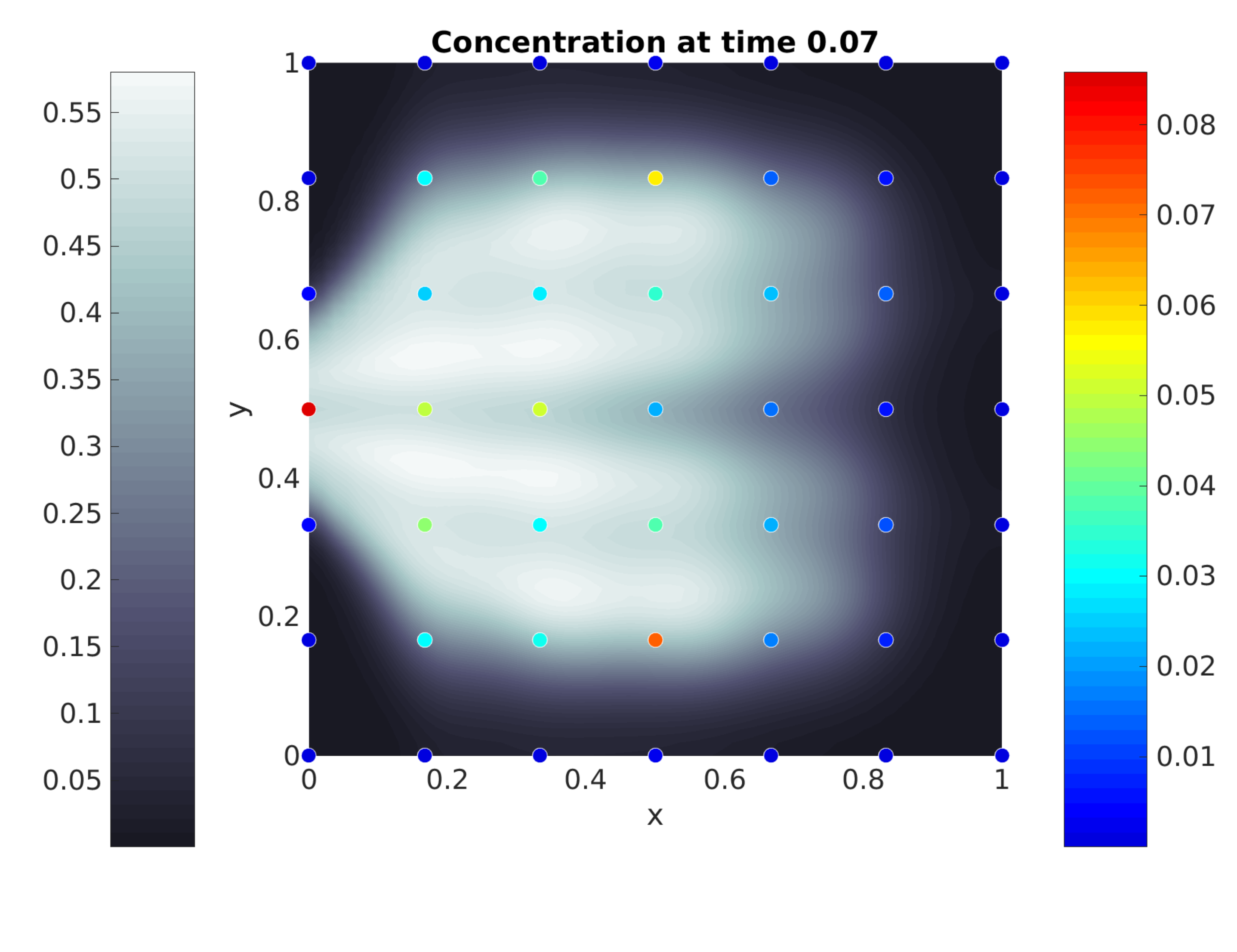}
	\includegraphics[width=0.47\linewidth]{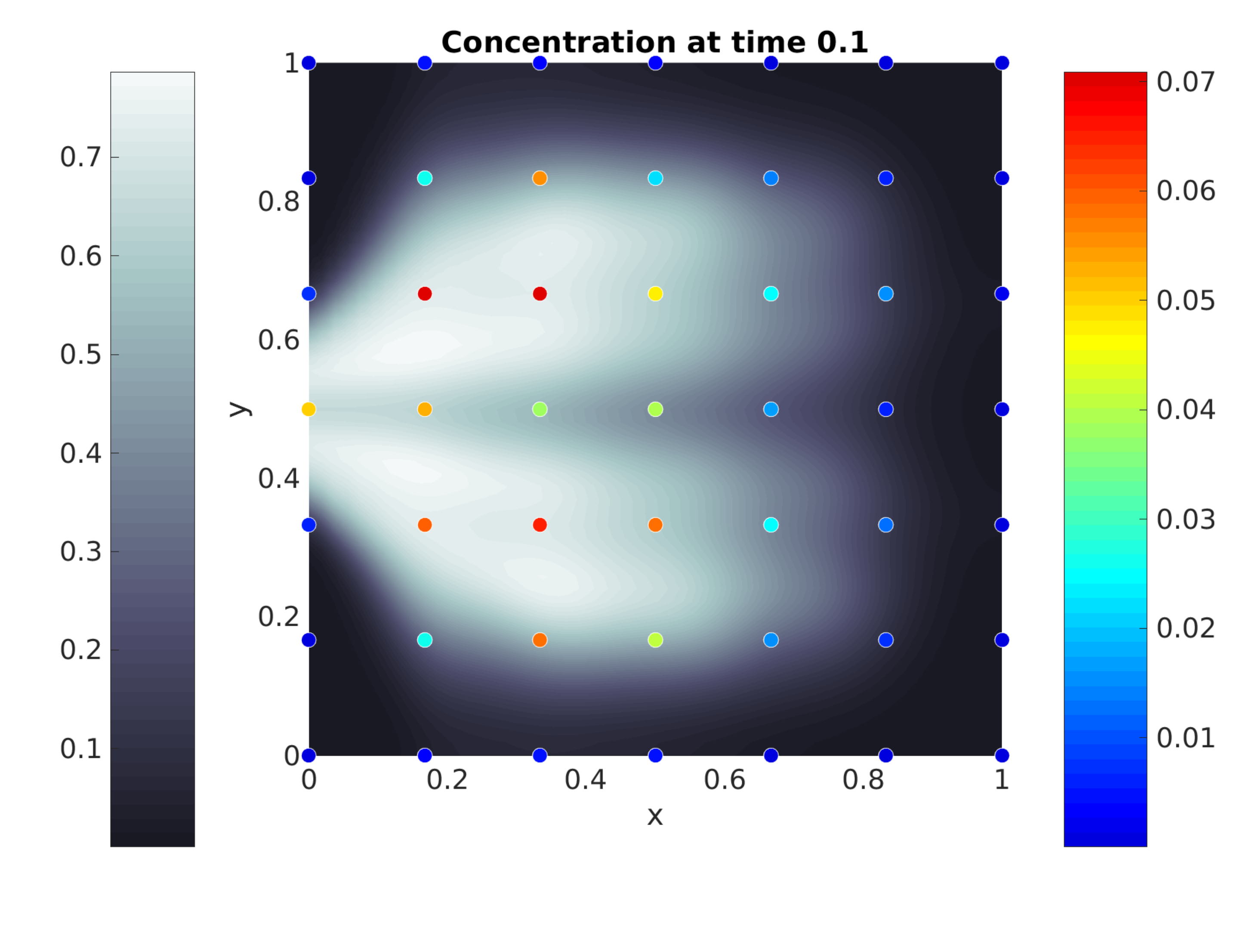}
	\includegraphics[width=0.47\linewidth]{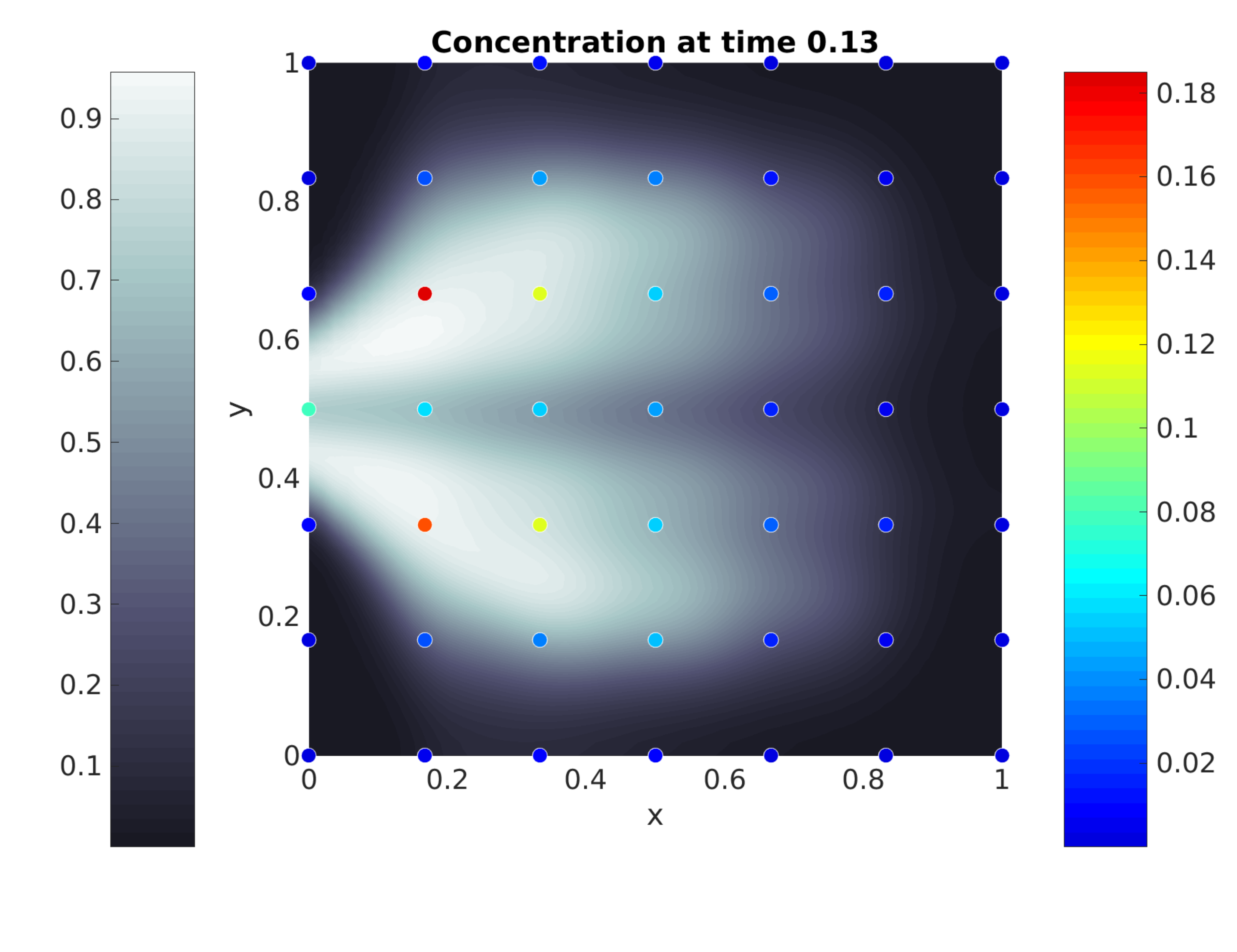}
	\includegraphics[width=0.47\linewidth]{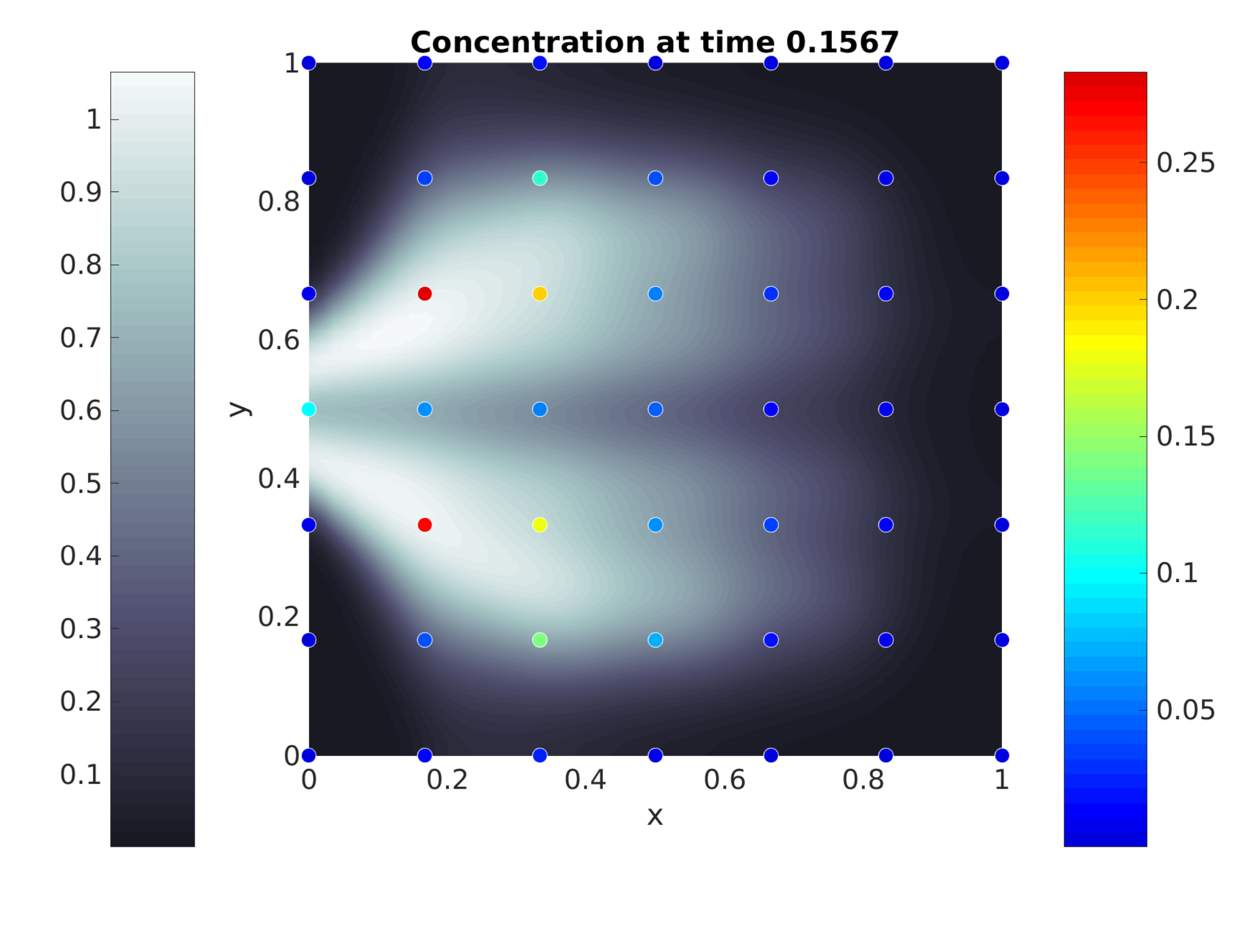}
	\caption{Concentration sensitivities at times 0.01, 0.04, 0.07, 0.10, 0.13, and 
                 0.1567.}
	\label{fig:Contaminant_sensitivities}
\end{figure}
\FloatBarrier
\clearpage

Note that both the sensitivity and concentration 
colorbar scales change in each plot, which is to allow the reader to visualize 
the results more clearly. 
We make the following observations about the concentration sensitivities.
(1)
As a general trend, the sensitivity of concentration increases in
time.  This is because the continuous source injection increases the amount of
tracer in the domain as time progresses, making concentration sensors
increasingly important. 
(2) We also see that as the mass of high tracer concentration
(depicted by bright white in the color map)
moves, sensors that observe this change in mass have increased importance
while the mass moves toward or away from the sensor and then 
decrease in importance after the  mass has moved passed. This phenomenon 
is particularly noticeable from the sensors in the high permeability channel 
at $y = 0.5$.
As the tracer mass moves from right to left at $y = 0.5$ the sensor's importance
increases following the back edge of the mass, and then decreases after the mass moves past. 
(3) We also notice that the sensor on the left boundary at $y = .5$
is highly important early in time and the sensors at $(0.2, 0.33)$ and $(0.2,0.66)$ are more important 
later in time. This is because the
majority of the tracer is getting advected toward those sensors which are therefore observing a large amount of tracer flow. 

To further study the concentration sensitivities, Figure~\ref{fig:sensitivities_timescatter} depicts the time evolution of each concentration sensitivity in
an array of plots. Each plot describes the time evolution of a single contaminant 
sensitivity and they are spatially arranged to correspond to the sensor locations they depict (compare with Figure~\ref{fig:source_loc} for their spatial association).

\begin{figure} [ht!]
	\centering
	\includegraphics[width=0.5\linewidth]{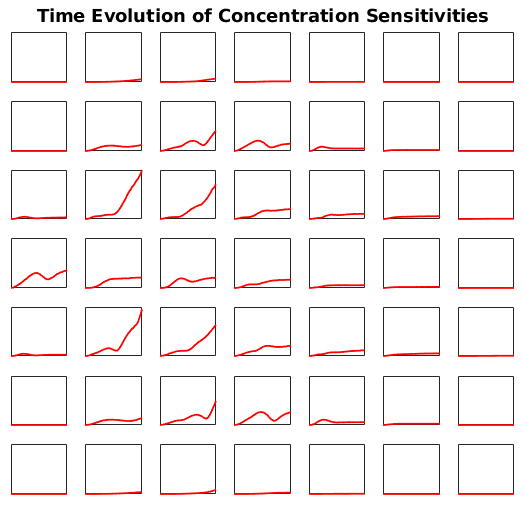}
	\caption{A spatial distribution of the time evolution of contaminant sensitivities. Each subplot has the same horizontal axis range depicting time from 0 to .16. Each vertical axis subplot has the same range depicting sensitivity from 0 to .2832}
	\label{fig:sensitivities_timescatter}
\end{figure}
\FloatBarrier

From Figure~\ref{fig:sensitivities_timescatter} we can see that some of the sensitivities decrease 
in importance over time, or have a range of time during which they decrease before 
they begin to increase again. Sensors for which the tracer permanently moves away 
experience a long-term decrease in importance.
The decrease in sensitivity at a sensor location, followed by a subsequent increase is likely
caused by the movement of tracer through the high permeability 
region at $y = 0.5$. As the tracer empties out of this central region, the concentration changes from being a 
single mass to being two distinct masses, one in the upper portion and the other in 
the lower portion of the domain. This splitting of the concentration mass,
when observed by a nearby sensor, likely causes a minor disturbance in the general 
trend of the sensitivity. 

Because this problem is advection dominated, the movement of the tracer through the domain 
has a large impact on the interpretation of the pressure sensitivities as well. Figure 
\ref{fig:Pressure_sensitivities} shows the spatial distribution of pressure sensitivities 
(depicted by colored points using the right colorbar scale) overlaid
atop the pressure field (depicted by a greyscale pressure map using the left 
colorbar scale). 

\begin{figure} [ht!]
	\centering
	\includegraphics[width=0.48\linewidth]{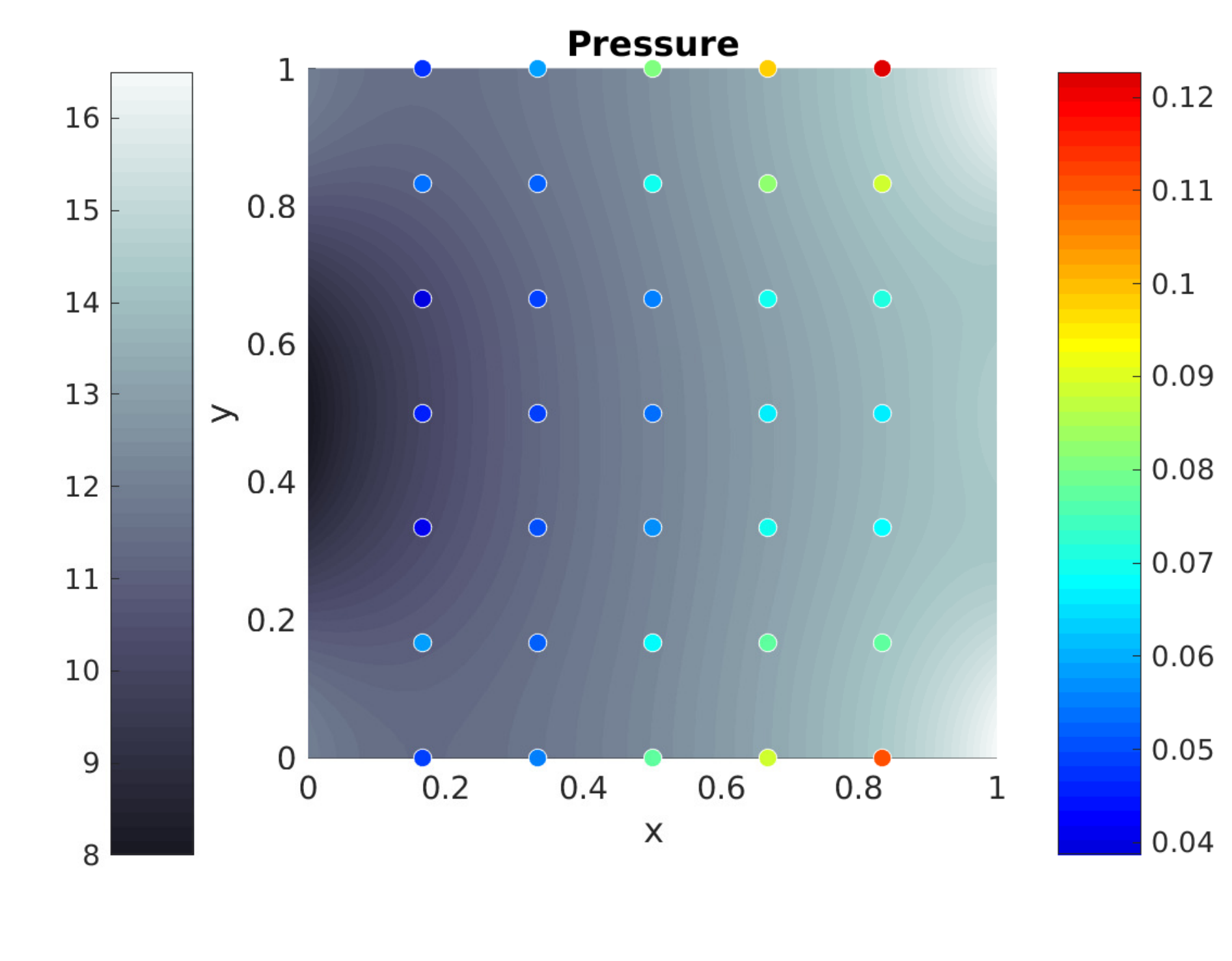}
	\caption{Pressure sensitivities}
	\label{fig:Pressure_sensitivities}
\end{figure}
\FloatBarrier

First, notice that the pressure sensitivities are larger than the concentration sensitivities at early time 
steps, but are eventually surpassed by the steadily increasing concentration sensitivities over time.  
Also observe that the sensors with highest pressure sensitivity, 
are in the upper right and lower right corners of the domain. Because the tracer moves
from right to left, the majority of information about the tracer advection is found
on the left side of the domain.
The right side of the domain, and particularly the low 
permeability regions at the top and bottom, have a lack of tracer flow information.
Thus, the inverse problem relies heavily on the pressure measurements in these regions 
to reconstruct
the permeability field, which corresponds to the higher pressure sensitivities in this region. 

\subsection{HDSA with Respect to Auxiliary Parameters} \label{subsec:aux_HDSA}
In this section, we study the pointwise 
hyper-differential sensitivities with respect to auxiliary parameters:
the source term and pressure Dirichlet boundary conditions. The diffusion coefficient is also an
auxiliary parameter, however according to the generalized sensitivities it is relatively unimportant so we will not investigate it further. 

We begin by
analyzing the sensitivities with respect to the source term. 
Figure~\ref{fig:source_sensitivities} depicts the sensitivities 
of the source at each injection well, next to a plot of the Darcy velocity field. 
The source uncertainty is discretized by taking a $3 \times 3$ mesh 
locally around each well. This models our uncertainty in the rate and distribution of the
injected tracer at each well due to hardware limitations, while the location of the injected tracer
is known.

\begin{figure} [ht!]
	\centering
	\includegraphics[width=0.48\linewidth]{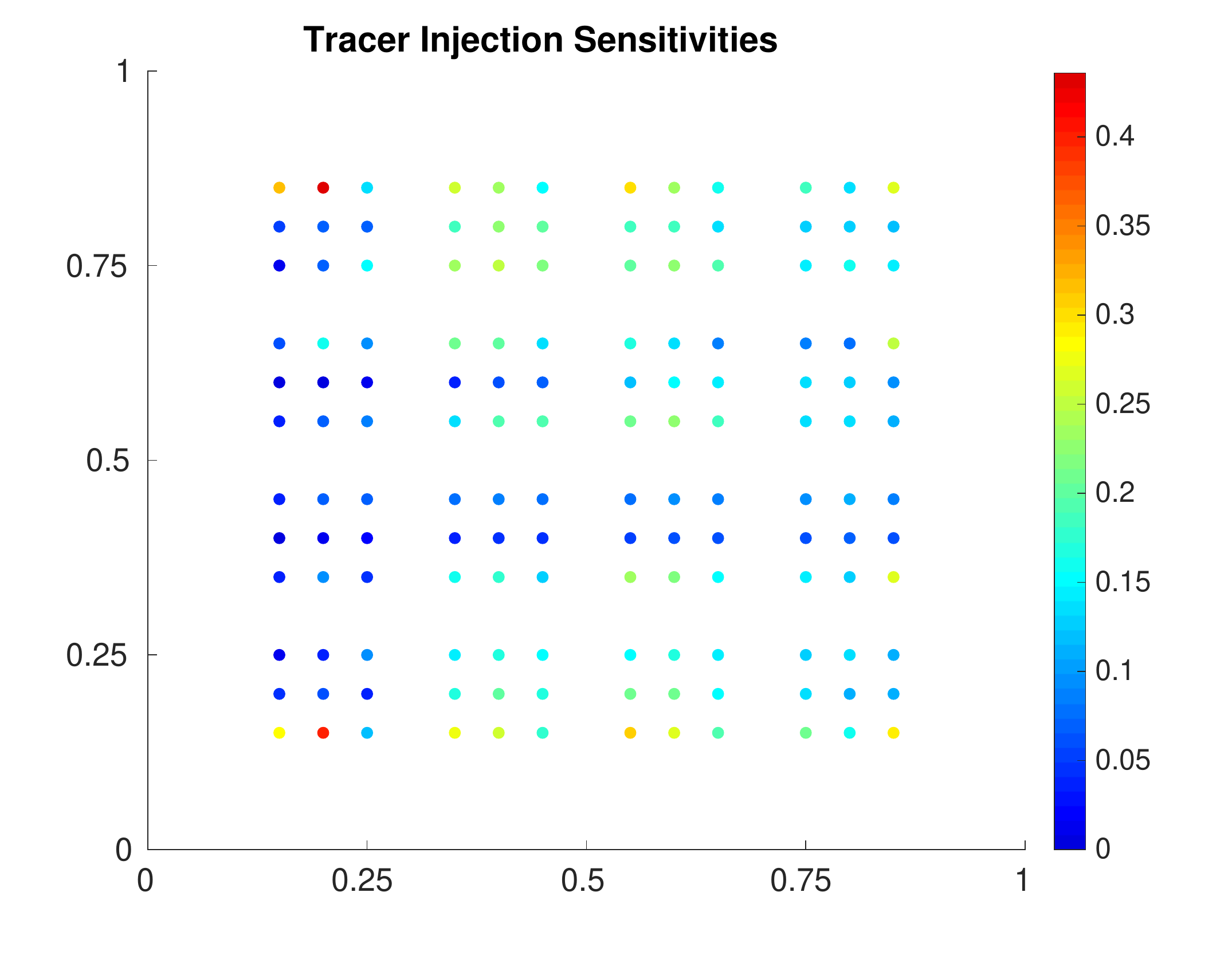}
	\includegraphics[width=0.51\linewidth]{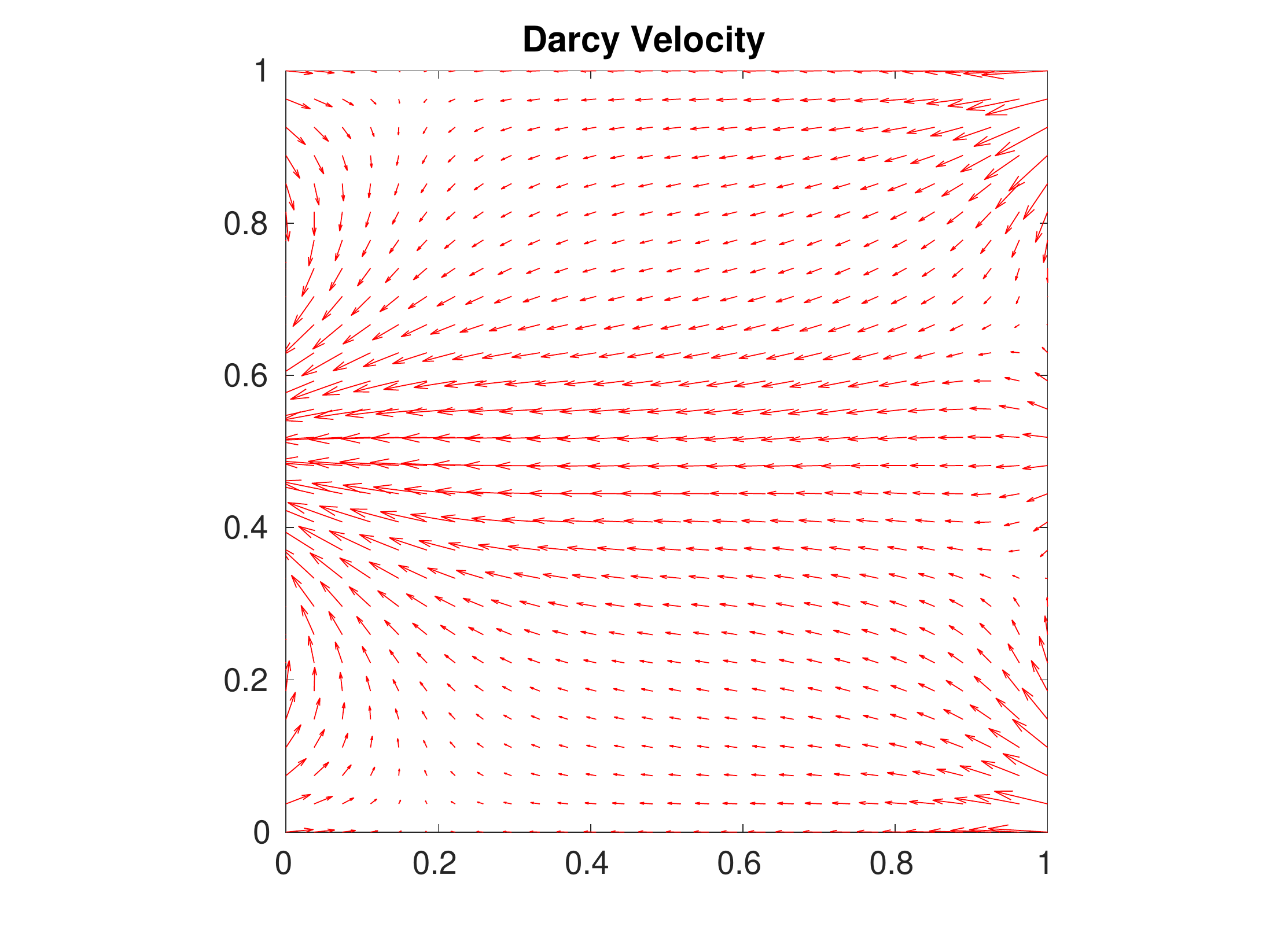}
	\caption{Left: Source sensitivities, \quad Right: Darcy Velocity Field}
	\label{fig:source_sensitivities}
\end{figure}
\FloatBarrier

From Figure~\ref{fig:source_sensitivities} we can see that the areas of high source sensitivity 
generally occur in regions with low Darcy velocity. This is likely because
in these regions, the problem is diffusion dominated and the tracer is not advected away from the injection site very quickly. Thus, if the source injection is perturbed, the
tracer will stay in that region and slowly diffuse, affecting concentration measurements in that area
for may time steps. If a source injection is perturbed in a region of high Darcy velocity, the 
tracer will be pulled away and mix with the rest of the tracer moving through the domain. 
Thus for this problem, the source injections have highest sensitivity in regions that are 
diffusion dominated.

According to the generalized sensitivities, the boundary conditions have the largest relative 
impact of any uncertain parameter on the solution. To further understand
the influence of the boundary conditions on the physical systems in the model, we consider
the hyper-differential sensitivities of the boundary conditions. Figure \ref{fig:BC_sensitivities}
depicts the sensitivities of the solution with respect to the pressure Dirichlet boundary 
conditions, discretized by 21 equally spaced nodes on each boundary. 

\begin{figure} [ht!]
	\centering
	\includegraphics[width=0.48\linewidth]{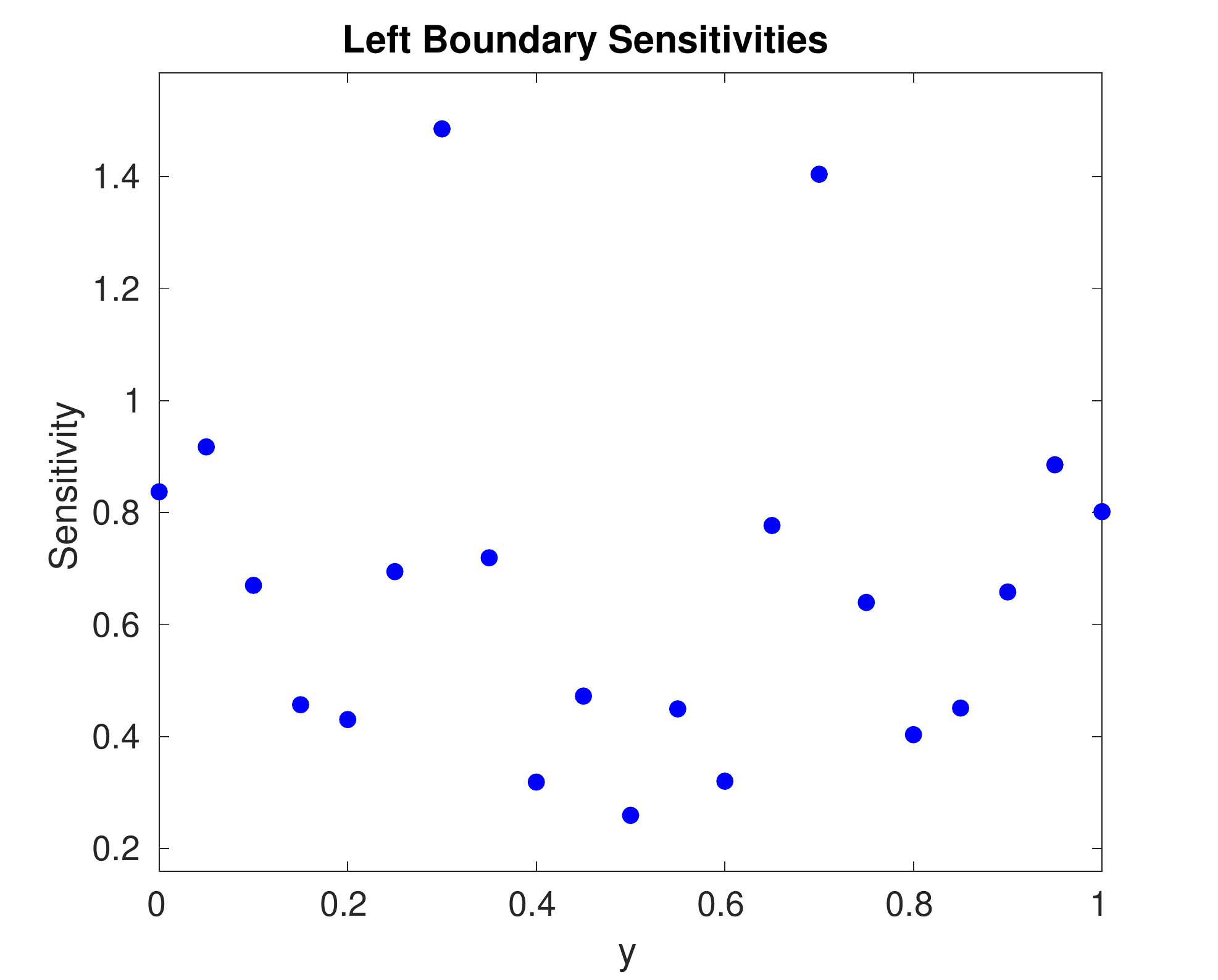}
	\includegraphics[width=0.48\linewidth]{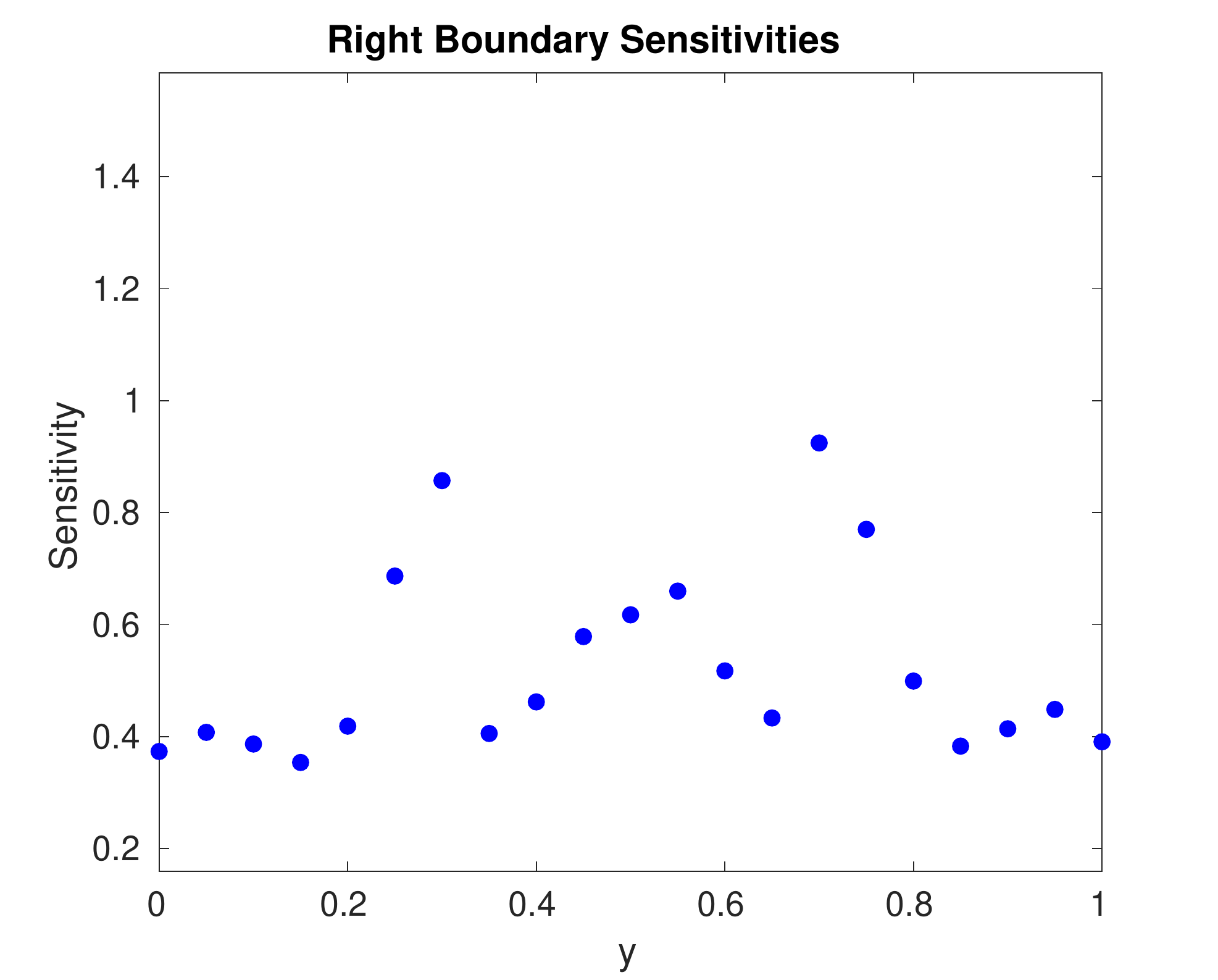}
	\caption{Pressure Dirichlet boundary condition sensitivities}
	\label{fig:BC_sensitivities}
\end{figure}
\FloatBarrier

From the boundary condition sensitivities in Figure~\ref{fig:BC_sensitivities} 
we can see that on both the left and right boundary, there is a heightened sensitivity 
around $y = 0.3$ and $y = 0.7$. This corresponds
to the area between the high permeability region through the middle of the domain, and the low 
permeability regions above and below it. A perturbation of the boundary conditions near the low 
permeability region is going to have minimal effect because the low permeability region 
is going to keep the Darcy velocity small relative to the rest of the domain regardless. 
Similarly, a perturbation of the boundary conditions in the high 
permeability region is going to have little effect because the high permeability is going to keep the
Darcy velocity relatively high in that area. 
A perturbation in the pressure boundary conditions will have maximal
impact in the thin region between the high and low permeability regions. This perturbation 
can cause the 
region of moderate Darcy velocity to become a region of either high or low Darcy velocity
relative to the other regions of the domain, significantly impacting the advection flow. 

\subsection{Discussion} \label{subsec:disc}
It is evident from our analysis that the hyper-differential sensitivities can provide a 
wealth of information about how the solution of the inverse problem depends on 
the interactions within the governing physics systems. These observations
are not readily apparent without the sensitivities, which emphasizes their usefulness in 
understanding the inverse problem. In particular, the sensitivities with respect to experimental
parameters can be used to determine where and when to place expensive, high fidelity sensors in
an experimental design, and where less accurate and more cost effective sensors can be uesd.

In addition to providing information about the underlying physics in a model,
these observations allow experimenters to determine how to design experiments, by prioritizing the measurement and estimation of all complementary parameters considered. 
For this problem, the generalized sensitivities inform us that tracer concentration is more
important to measure accurately than pressure, which informs the design of sensors and data 
collection techniques. We also learn that accurately estimating the pressure Dirichlet boundary
conditions is highly important, while the diffusivity coefficient is relatively unimportant. This
information informs the model specification and how these parameters should be estimated 
and considered in the model. 

\section{Conclusion}
\label{sec:conclusion}
In this article, we have presented a mathematical framework for
hyper-differential sensitivity analysis in the context of inverse problems
constrained by multiphysics systems of partial differential equations. The
mathematical formulation involves derivative based local sensitivity analysis
of the solution of an inverse problem with respect to perturbations of
parameters.  This framework is general and can be applied to a wide variety of
inverse problems. The usefulness of HDSA is most apparent in the context of
complicated multiphysics systems with many uncertain parameters. By introducing
sensitivity analysis with respect to experimental parameters and maturing the
generalized hyper-differential sensitivity indices, we have enabled analysis of the
relative importance of both auxiliary model parameters and experimental
parameters, such as data sources. Studying hyper-differential sensitivities
provides new insights into both understanding the underlying physical systems
of a model, and designing experiments to solve inverse problems.  In
particular, comparing the relative importance of spatially and temporally
distributed measurements with various sensor types provides unique insights
that are difficult to attain using traditional experimental design
methodologies. HDSA compliments experimental design by providing a systematic
way to compare multiphysics parameters and data sources.

HDSA is an emerging technology and as such there are several important considerations
to be studied in future work. One questions is 
``how robust are the sensitivities of experimental parameters to perturbations of design?"
In practice, one may not be able to place a sensor exactly where a design indicates it 
should be. If a
sensor's location is perturbed within a local area, how will this affect the magnitude of the
sensitivity, and that of its neighboring sensors? Ideally, perturbing the location of a sensor 
slightly will have a minimal impact on the sensitivity at that sensor and its neighbors, 
indicating that the hyper-differential sensitivities are robust to perturbations of the design,
but this has yet to be rigorously verified. 

Moreover, HDSA of inverse problems requires availability of measurement data.
For the computational results in this paper, we assumed that we had some set of
experimental data, but in many applications we would like to compute
sensitivities a priori, before data is collected.  In such cases, one could
generate training data by applying the forward model to a sample of the
inversion parameters drawn from a prior distribution, giving rise to a
distribution of the sensitivities.

\appendix
\section{Details for the model inverse problem~\eqref{equ:ip_discrete}}\label{mtd:adj_op}
First, we illustrate the computation of gradient and Hessian of $\hat{J}$. 
To facilitate this, we introduce the Lagrangian
\[
\mathcal{L}(\uu,\m,\lam,\thth) = \hat{J}(\m) + \lam^\top(\mat{L}(\m)\uu - \V\thth),
\]
where $\vec{\lambda}$ is a Lagrange multiplier (vector). 
To compute the gradient of $\hat{J}$, using the so called formal Lagrange 
approach, we consider the variations of $\mathcal{L}$ with respect to 
$\vec{\lambda}$, $\vec{u}$ and $\vec{m}$. Note that
\[
\LL_{\lam}(\uu,\m,\lam,\thth) = \mat{L}(\m)\uu - \V\thth \quad \text{and} \quad
\LL_{\uu}(\uu,\m,\lam,\thth)  = \dJdu + \lam^\top \mat{L}(\m). 
\]
Setting these variations equal to 0  
results in the state and the adjoint equations:
\[
\mat{L}(\m)\uu = \V\thth \quad\text{and}\quad \mat{L}(\m)^\top\lam = -\mat{Q}^\top\mat{W}(\mat{Q}\uu-
\vec{y}).
\]
Then, the gradient $g(\vec m)$ of $\hat{J}$ satisfies, 
$g(\vec m)^T = \LL_{\m}(\uu,\m,\lam,\thth) = \dJdm + \lam^\top\dAdm\uu$; therefore,
\[
g(\m) = \alpha\mat{R}\vec{m} + \mat{C}^\top\lam \quad \text{with} \quad \mat{C} = 
\frac{\partial}{\partial \m}\big(\mat{L}(\m)\uu\big).
\]
To compute the action of the Hessian $\mat{H}(\m)$ of $\hat{J}$ (at $\vec{m}$) on a vector $\widehat\m$, 
we differentiate through the 
directional derivative $\langle g(\vec{m}), \widehat{\m}\rangle$. This is faciliated by introducing the 
``meta-Lagrangian'':
\begin{multline*}
\LL^H(\uu,\m,\lam,\thth,\widehat{\uu},\widehat{\lam};\widehat{\m}) = 
\widehat{\m}^\top\big[\mat{C}^\top\lam +\alpha\mat{R}\m\big] + 
\widehat{\lam}^\top\big[\mat{L}(\m)\uu - \V\thth\big] + \\ 
\widehat{\uu}^\top\big[\mat{L}(\m)^\top\lam + \mat{Q}^\top\mat{W}(\mat{Q}\uu- \vec{y})\big].
\end{multline*}
The Lagrange multipliers $\widehat{\vec{u}}$ and $\widehat{\vec{\lambda}}$ are refered to as the 
incremental state and adjoint variables; see e.g.,~\cite{Petra11}.
Letting the variations of $\mathcal{L}^H$ with respect to $\widehat{\vec{\lambda}}$ 
and $\widehat{\vec{u}}$ vanish gives 
\begin{align}
\mat{L}(\m)\widehat{\uu} &= -\mat{C}\widehat{\m}, \quad &\text{(incremental 
state equation)}\label{equ:inc_state}\\
\mat{L}(\m)^\top\widehat{\lam} &= -\LL_{\uu\m}\widehat{\m} - \LL_{\uu\uu}\widehat{\uu}.
\quad &\text{(incremental adjoint equation)}
\label{equ:inc_adjoint}
\end{align}
The Hessian apply is then given by, $\left[\mat{H}(\m)\widehat{\m}\right]^\top = \LL_{\m}^H\widehat\m$, resulting in
\begin{equation}\label{equ:hess_apply}
\mat{H}(\m)\widehat{\m} = \mathcal{L}_{\m\m}\widehat{\m} + \LL_{\m\uu}\widehat{\uu} + \mat{C}^\top\widehat{\lam}.
\end{equation}
Note that in the above equations
\[
    \mathcal{L}_{\uu\uu} = \mat{Q}^\top\mat{W}\mat{Q}, \quad
    \mathcal{L}_{\m\m} = \alpha \mat{R} + \frac{\partial}{\partial \m}(\mat{C}^T \vec{\lambda}),
    \quad \text{and} \quad \mathcal{L}_{\m\uu} = \mathcal{L}_{\uu\m}^\top=\frac{\partial}{\partial \m}(\mat{L}(\m)^T \vec{\lambda})
\]
We summarize the compuation of $\mat{H}(\m)\widehat{\m}$ in Algorithm~\ref{alg:hess_apply_alg}.
\begin{algorithm}
\caption{Computation of $\mat{H}(\m)\widehat{\m}$ for a given $\widehat{\m}$.}
\label{alg:hess_apply_alg}
\begin{algorithmic}
\STATE solve the incremental state equation~\eqref{equ:inc_state} for $\widehat{\uu}$
\STATE Solve the incremental adjoint equation~\eqref{equ:inc_adjoint} for $\widehat{\lam}$
\STATE Evaluate $\mat{H}(\m)\widehat{\m}$ accordng to~\eqref{equ:hess_apply}.
\end{algorithmic}
\end{algorithm}
Note that the cost associates with Algorithm~\ref{alg:hess_apply_alg} is 
two (linear) PDE solves.
Also, by replacing the expressions for the incremental state and adjoint variables 
in the expression for the Hessian apply, we can write the (reduced) Hessian as:
\[
    \mat{H} = \mat{C}^\top \mat{L}(\m)^{-\top}\mathcal{L}_{\uu\uu}
              \mat{L}(\m)^{-1}\mat{C} + \mathcal{L}_{\m\m} 
              - \mathcal{L}_{\m\uu} \mat{L}(\m)^{-1}\mat{C}
              - \mat{C}^\top \mat{L}(\m)^{-\top}\mathcal{L}_{\uu\m}.
\]

Letting $\mat{B}$ be the (discretized) Fr\'echet derivative of the gradient with respect to $\thth$, 
$$ \mat{B} = g_{\thth}(\m) = -\mat{C}^\top\mat{L}(\m)^{-\top}\mat{Q}^\top\mat{W}\mat{Q}\mat{L}(\m)^{-1}\mat{V} + \mathcal{L}_{\m\uu}\mat{L}(\m)^{-1}\mat{V}. $$ 
This is the discretized version of the operator $\mathcal{B}$ in~\eqref{sen_op}.

\section*{Acknowledgements}
This paper describes objective technical results and analysis. Any
subjective views or opinions that might be expressed in the paper do
not necessarily represent the views of the U.S. Department of Energy
or the United States Government. Sandia National Laboratories is a
multimission laboratory managed and operated by National Technology
and Engineering Solutions of Sandia LLC, a wholly owned subsidiary of
Honeywell International, Inc., for the U.S. Department of Energy's
National Nuclear Security Administration under contract
DE-NA-0003525. SAND2020-2298 J.

The work of I. Sunseri and A. Alexanderian was supported in part by 
National Science Foundation under grant DMS-1745654.

\section*{Bibliography}
\bibliographystyle{unsrt}

\bibliography{refs}

\end{document}